\documentclass[12pt]{amsart}
\usepackage[dvips]{epsfig}
\usepackage{graphics}
\usepackage{latexsym}
\usepackage{verbatim}
\usepackage{amsmath}
\usepackage{amsthm}
\usepackage{amssymb}
\usepackage{hyperref}
\usepackage[]{hyperref}
\usepackage{float}
\hypersetup{
    colorlinks=true,%
    filecolor=black,%
    linkcolor=black,%
    urlcolor=black  %
}
\usepackage [table]{xcolor}
\usepackage{multirow}
\usepackage{float}
\usepackage{tikz}
\usepackage{subfig}

\graphicspath{{images/}} 
\newtheorem{theorem}{Theorem}[section]
\newtheorem{lemma}[theorem]{Lemma}

\newtheorem{proposition}[theorem]{Proposition}

\newtheorem{remark}[theorem]{Remark}

\newcommand\CC{\hbox{C\kern -.58em {\raise .54ex \hbox{$\scriptscriptstyle |$}}\kern-.55em {\raise .53ex \hbox{$\scriptscriptstyle |$}}}}
\newcommand\NN{\hbox{I\kern-.2em\hbox{N}}}
\newcommand\RR{\mathbb{R}}

\newcommand\ZZ{{{\rm Z}\kern-.28em{\rm Z}}}

\newcommand\Gradxp{ \nabla_{\mathbf{x}}}

\newcommand\Gradvp{ \nabla_{\mathbf{v}}}
\newcommand\Div{ \textrm{div}}

\newcommand\thDiv{ \textrm{\emph{div}}}
\newcommand\thRot{ \textrm{\emph{curl}}}

\newcommand\ds{ \displaystyle }
\newcommand\vp{ \varphi }
\renewcommand\d{\partial}
\newcommand\dd{\textrm{d}}
\newcommand\thdd{\textrm{\emph{d}}}

\newcommand\bB{{\mathbf B}}

\newcommand\bE{{\mathbf E}}
\newcommand\bF{{\mathbf F}}

\newcommand\bU{{\mathbf U}}
\newcommand\bV{{\mathbf V}}

\newcommand\bX{{\mathbf X}}
\newcommand\bY{{\mathbf Y}}

\newcommand\bv{{\mathbf v}}
\newcommand\bw{{\mathbf w}}
\newcommand\bx{{\mathbf x}}
\newcommand\by{{\mathbf y}}

\def\eps{\varepsilon }

\def\signFF{\bigskip\bigskip\hspace{80mm}
\vbox{{\sc Francis Filbet\par\vspace{3mm}
Universit\'e de Toulouse III \& IUF \par
UMR5219, Institut de Math\'ematiques de Toulouse,\par
118, route de Narbonne\par
F-31062 Toulouse cedex,  FRANCE
\par\vspace{3mm}e-mail:} francis.filbet@math.univ-toulouse.fr }}

\def\signLMR{\bigskip\bigskip\hspace{80mm}
\vbox{{\sc  Luis Miguel Rodrigues \par\vspace{3mm}
Universit\'e de Rennes 1,\par
UMR6625, IRMAR,\par
263 avenue du General Leclerc,\par
F-35042 Rennes Cedex,  FRANCE
\par\vspace{3mm}e-mail:}  luis-miguel.rodrigues@univ-rennes1.fr }}

\begin{document}

\title[PIC methods for inhomogenous strongly magnetized
plasmas]{Asymptotically preserving particle-in-cell methods for inhomogenous strongly magnetized
plasmas}

\author{Francis Filbet, Luis Miguel Rodrigues}

\maketitle

\begin{abstract}
We propose a class of Particle-In-Cell (PIC) methods for the Vlasov-Poisson system with a strong and inhomogeneous external magnetic field with fixed direction, where we
focus on the motion of particles in the plane orthogonal to the
magnetic field (so-called poloidal directions). In this regime, the time step can
be subject to stability constraints related to the smallness of
Larmor radius and plasma frequency. To avoid this limitation, our approach is based on first and higher-order semi-implicit numerical
schemes already validated on dissipative systems \cite{BFR:15} and for
homogeneous magnetic fields \cite{FR16}. Thus, when the magnitude of the external magnetic
field becomes large, this method provides a consistent PIC
discretization of the guiding-center system taking into account variations of the magnetic field. We carry out some theoretical proofs and perform several numerical experiments that establish a solid validation of the method and its underlying concepts.
 \end{abstract}

\vspace{0.1cm}



\section{Introduction} 
\setcounter{equation}{0}
\label{sec:1}
We  consider a magnetic field acting in the vertical direction and
only depending on time $t\geq 0$ and the first two components $\bx=(x_1,x_2)\in \RR^2$,
that is,  
$$
\bB(t,\bx) \,\,=\,\, \frac{1}{\eps} \, \left( \begin{array}{l}0\\0\\b(t,\bx)\end{array}\right), 
$$
with  $b(t,\bx)>0$, where $\varepsilon>0$ is a small parameter related to the ratio between the reciprocal Larmor frequency and the advection time scale. This magnetic field is
applied to confine a plasma constituted of a large number of charged particles, modelled by a distribution function $f^\eps$ solution to the Vlasov equation coupled with the Poisson equation satisfied by the self-consistent electrical potential $\phi^\eps$. Here we focus on the long time behavior of the plasma in a plane orthogonal to the external magnetic field. Therefore we consider the following form of the two dimensional Vlasov-Poisson system with an external strong magnetic field
\begin{equation}
\label{eq:vlasov2d}
\left\{
\begin{array}{l} 
\displaystyle{\varepsilon\frac{\partial f^\eps}{\partial t}\,+\,\mathbf{v}\cdot\Gradxp f^\eps \,+\,\left(
  \mathbf{E}^\eps(t,\bx) \,-\, b(t,\bx)\frac{\bv^\perp}{\eps} \right)\cdot\Gradvp f^\eps
\,=\, 0,}
\\
\,
\\
\displaystyle{\bE^\eps = - \Gradxp \phi^\eps, \quad -\Delta_{\bx} \phi^\eps =
  \rho^\eps,\quad \rho^\eps=\int_{\RR^2} f^\eps d\bv,}
\end{array}\right.
\end{equation}
where we use notation $\bv^\perp=(-v_2,v_1)$ and for simplicity we set all immaterial physical constants to one. The term $\varepsilon$ in
front of the time derivative of $f^\eps$ stands for the fact that we want to follow
the solution on times sufficiently large to observe non trivial averaged evolutions.

We want to  construct numerical solutions to the Vlasov-Poisson system \eqref{eq:vlasov2d} by particle methods (see \cite{birdsall}), which consist in
approximating the distribution function by a finite number of
macro-particles. The trajectories of these particles are determined from
the characteristic curves corresponding to the Vlasov equation
\begin{equation}
\label{traj:00}
\left\{
\begin{array}{l}
\ds{\varepsilon\frac{\dd\bX^\eps}{\dd t} \,=\, \bV^\eps,} 
\\
\,
\\
\ds{\varepsilon\frac{\dd\bV^\eps}{\dd t} \,=\, -\frac{1}{\varepsilon} \,{b}(t,\bX^\eps)\,(\bV^\eps)^{\perp}    \,+\, \bE^\eps(t,\bX^\eps), }
\\
\,
\\
\bX^\eps(t^0) = \bx^0, \, \bV^\eps(t^0) = \bv^0,
\end{array}\right.
\end{equation}
where we use the conservation of $f^\eps$ along the characteristic curves  
$$
f^\eps(t, \bX^\eps(t),\bV^\eps(t)) = f^\eps(t^0, \bx^0, \bv^0)
$$
and the electric field is computed from a discretization of the Poisson equation
in (\ref{eq:vlasov2d}) on a mesh of the physical space.
 
Before describing and analyzing a class of numerical methods for the
Vlasov-Poisson system \eqref{eq:vlasov2d} in the presence of a strong external inhomogeneous magnetic field, we first briefly expound what may be expected from the continuous model in the limit $\eps\to0$.

\subsection{Asymptotics with external electromagnetic fields}\label{s:external-asymptotics}

To gain some intuition, we first discuss the case when all electromagnetic fields $(\bE^\eps, b)$ are known. Conclusions of the present section may actually be completely justified analytically. Yet we do not pursue this line of exposition here to keep the discussion as brief and tight as possible.

Explicitly, in the present section we consider given (independent of $\eps$) smooth electromagnetic fields  $(\bE^\eps, b)$ 
and we 
assume that $b$ is bounded below away from zero, that is, we assume that there exists $b_0>0$ such that
\begin{equation}
\label{hyp:1}
 b(t,\bx) \geq b_0 \quad \forall \,(t,\bx)\in \RR^+\times\RR^2.
\end{equation}

In the limit $\eps\to0$ one expects oscillations occurring on typical time scales $O(1/\eps^2)$ to coexist with a slow dynamics evolving on a time scale $O(1)$. We sketch now how to identify a closed system describing at main order the slow evolution. To begin with, note that from the second line of system~\eqref{traj:00} it does follow that $\bV^\eps$ oscillates at order $1/\eps^2$ thus remains bounded and converges weakly\footnote{Though we do not want to be too precise here, let us mention that in the present discussion \emph{weakly} and \emph{strongly} refer to the weak-* and strong topologies of $L^\infty$ and that the weak convergences that we encounter actually correspond to strong convergence in $W^{-1,\infty}$.} to zero. As we detail below, one may also combine both lines of the system to obtain
\begin{equation}\label{res:0bis}
\frac{\dd}{\dd t}\left( \bX^\eps - \eps \,\frac{(\bV^\eps)^\perp}{b(t,\bX^\eps)}\right)  \,=\,  \ \bF(t,\bX^\eps) 
\,-\,(\bV^\eps)^\perp\left[
\,\dd_\bx \left(\frac1b\right)(t,\bX^\eps)(\bV^\eps)
\,+\,\eps\frac{\d}{\d t}\left(\frac1b\right)(t,\bX^\eps) \,\right]
\end{equation}
where the force field $\bF$ is the classical electric drift, given by
$$
\bF(t,\bx) \,:=\, -\frac{1}{b(t,\bx)}\, \bE^{\perp}(t,\bx),\quad \forall \,(t,\bx)\in \RR^+\times\RR^2\,.
$$  
Indeed
\begin{eqnarray*}
\,\frac{\dd}{\dd t} \bX^\eps &=& 
\frac1\eps\bV^\eps
\,=\,
\frac{\eps}{b(t,\bX^\eps)}\left(\frac{\dd\bV^\eps}{\dd t}\right)^\perp+\,\bF(t,\bX^\eps)\\
&=& \eps \frac{\dd}{\dd t} \left(\frac{(\bV^\eps)^\perp}{b(t,\bX^\eps)}\right)+\,\bF(t,\bX^\eps)
-\,\eps(\bV^\eps)^\perp\left(\frac{\d}{\d t}\left(\frac1b\right)(t,\bX^\eps) + \dd_\bx \left(\frac1b\right)(t,\bX^\eps) \,\left(\frac{\dd\bX^\eps}{\dd t}\right)\right)\\
&=& \eps \frac{\dd}{\dd t} \left(\frac{(\bV^\eps)^\perp}{b(t,\bX^\eps)}\right)+\,\bF(t,\bX^\eps)
-\,(\bV^\eps)^\perp\left(\eps\frac{\d}{\d t}\left(\frac1b\right)(t,\bX^\eps) + \dd_\bx \left(\frac1b\right)(t,\bX^\eps) \,(\bV^\eps)\right).
\end{eqnarray*}
This shows that $\bX^\eps$ evolves slowly but, as such, does not provide a closed asymptotic evolution in the limit $\eps\to0$. Indeed, except in the case when $b$ is constant dealt with in \cite{FR16} we also need to know what happens to expressions that are quadratic in $\bV^\eps$ and this does not follow readily from the weak convergence of $\bV^\eps$.

\begin{remark}\label{rk:homogeneous}
Incidentally note that, in contrast, when $b$ is constant, as in \cite{FR16}, the formulation of equation~\eqref{res:0bis} is already sufficient to identify a guiding center evolution and obtain that $\bX^\eps$ converges to $\bY$ solving
\begin{equation}\label{gc:-1}
\frac{\dd\bY}{\dd t} \,=\,  \bF(t,\bY)\,.
\end{equation}
\end{remark}

The missing piece of information may be intuited from
\begin{lemma}
\label{lmm:01}
Consider $\bX^\eps=(x_1^\eps,x_2^\eps)$ and
$\bV^\eps=(v_1^\eps,v_2^\eps)$ solving \eqref{traj:00}. Then we have
\begin{equation}
\frac{\thdd}{\thdd t} \left( \,\frac{1}{2}\,\|\bV^\eps\|^2 \,-\, \eps \,\bF(t,\bX^\eps)\cdot
    \bV^\eps\,\right)  \,\,=\, \, -\bV^\eps \,\cdot\,\thdd_\bx \bF(t,\bX^\eps)\,(\bV^\eps)
\,\,-\,\, \eps  \,\bV^\eps\cdot\frac{\d \bF}{\d t}(t,\bX^\eps),
\label{eq:modulus} 
\end{equation}
and with $\bE=(E_1,E_2)$
\begin{equation}
\label{eq:VV}
\left\{
\begin{array}{l}
\ds\frac{1}{2}\,\frac{\thdd}{\thdd t} \left( \,|v^\eps_1|^{2} -  |v^\eps_2 |^{ 2}\,\right) 
\,=\, 
\frac{b(t,\bX^\eps)}{\eps^2} \,2v^\eps_1 \,v^\eps_2
\,+\,\frac{v^\eps_1\,E_1(t,\bX^\eps) - v_2\,E_2(t,\bX^\eps)}{\eps},
\\
\,
\\
\ds\frac12\frac{\thdd}{\thdd t} \left(\, 2v^\eps_1\, v^\eps_2  \;\right) \,=\,
-\, \frac{b(t,\bX^\eps)}{\eps^2} \,\left(\,|v^\eps_1|^{2}\,-\,|v^\eps_2|^{2}\,\right)
\,+\,\frac{v_2^\eps\,E_1(t,\bX^\eps) + v_1^\eps\,E_2(t,\bX^\eps)}{\eps}\,.
\end{array}\right.
\end{equation}
\end{lemma}
\begin{proof}
First note that the second equation in \eqref{traj:00} may also be written as
$$
 \frac{\bV^\eps}{\eps} \,=\, \bF(t,\bX^\eps) \,+\,\frac{\eps}{b(t,\bX^\eps)}\,\left(\frac{\dd\bV^\eps}{\dd t}\right)^\perp\,. 
$$
Therefore
\begin{eqnarray*}
\frac{1}{2}\,\frac{\dd}{\dd t} \|\bV^\eps\|^2 &=& 
\bV^\eps\,\cdot \frac{\dd\bV^\eps}{\dd t}
\,=\,\eps\,\bF(t,\bX^\eps) \cdot
\frac{\dd\bV^\eps}{\dd t}\\
&=& \eps \frac{\dd}{\dd t} \left(\bF(t,\bX^\eps)\cdot \bV^\eps\right) 
- \eps\,\bV^\eps\cdot \left(\frac{\d\bF}{\d t}(t,\bX^\eps) + \dd_\bx \bF(t,\bX^\eps) \, \frac{\dd\bX^\eps}{\dd t} \right)\\
&=& -\bV^\eps\cdot \dd_\bx \bF(t,\bX^\eps) \,(\bV^\eps)\;+\,\eps \left(\,\frac{\dd}{\dd t} \left[\bF^\eps(t,\bX^\eps)\cdot \bV^\eps\right] \,- \,\bV^\eps\cdot \frac{\d\bF}{\d t}(t,\bX^\eps)\,\right).
\end{eqnarray*}
Hence \eqref{eq:modulus}. Likewise we may obtain the first equation in \eqref{eq:VV} --- for $(|v^\eps_1|^2 -|v^\eps_2|^2)$ --- by multiplying the second equation of \eqref{traj:00} by $(v^\eps_1,-v^\eps_2)/\eps$ and the equation for $v^\eps_1\,v^\eps_2$ by multiplying it by
$(v^\eps_2,v^\eps_1)/\eps$.
\end{proof}

This crucial lemma suggests that 
\begin{itemize}
\item the  microscopic kinetic energy  $\frac{1}{2}\,\|\bV^\eps\|^2$ also evolves on a slow scale; 
\item the two variables  $2v^\eps_1\,v^\eps_2$ and $|v^\eps_1|^{2}\,-\,|v^\eps_2|^{2}$ oscillate at scale $1/\eps^2$ thus converge weakly to zero. 
\end{itemize}
Indeed, the couple $\left(2v_1^\eps\,v_2^\eps\,,\,|v_1^\eps|^{2}\,-\,|v_2^\eps|^{2}\right)$ solves a forced spring-mass system with non constant frequency of
oscillation of typical size $1/\eps^2$. Since the foregoing three variables generates all expressions quadratic in $\bV^\eps$, this is indeed sufficient to conclude.
  
From previous considerations it follows that when $(\bX^\eps,\bV^\eps)$ solves \eqref{traj:00} the couple $(\bX^\eps,\frac{1}{2}\|\bV^\eps\|^2)$ converges strongly to some $(\bY,g)$ as $\eps\to0$. To identify an evolution for the limiting $(\bY,g)$, we write relevant quadratic terms of \eqref{res:0bis} and \eqref{eq:modulus} in terms of quantities of Lemma~\ref{lmm:01}. First  
\begin{eqnarray*}
-\bV^{\eps}\,\dd_\bx\left(\frac1b\right)(t,\bX^\eps)\,(\bV^\eps)&=& -\frac{1}{2}
\|\bV^\eps\|^2 \,\nabla_\bx \left(\frac1b\right)(t,\bX^\eps) \,\,-\,\,  \frac{1}{2}
\left[ |v_1^{\eps}|^2 -|v_2^{\eps}|^2  \right]\,\left(\begin{array}{l} \d_{x_1}\\-\d_{x_2}\end{array}\right) \left(\frac1b\right)(t,\bX^\eps)
\\
&&-\,\,
\left[ v_1^{\eps} \,v_2^{\eps}  \right]\,
\left(\begin{array}{l} \d_{x_2}\\\d_{x_1}\end{array}\right)\left(\frac1b\right)(t,\bX^\eps)
\end{eqnarray*}
converges weakly to
$$
-g\,\nabla_\bx \left(\frac1b\right)(t,\bY)
$$
as $\eps\to0$. Thus, taking the limit $\eps\to0$ in \eqref{res:0bis} gives 
\begin{equation}
\frac{\dd\bY}{\dd t} \,=\,  \bF(t,\bY)   \,+\, g \,\frac{\nabla_\bx^\perp b}{b^2}(t,\bY). 
\label{gc:1}
\end{equation}
Likewise, $\bV^\eps \,\cdot\,\dd_\bx\bF(t,\bX^\eps)\,(\bV^\eps)$ converges weakly to $g\,\Div_\bx(\bF)(t,\bY)$ so that
\begin{equation}
\frac{\dd g}{\dd t} \,=\,  -\Div_\bx(\bF)(t,\bY)\,g.
\label{gc:2}
\end{equation}

Gathering the latter results, we get that $(\bX^\eps,\frac{1}{2}\|\bV^\eps\|^2)$ converges strongly to $(\bY,g)$ solving
\begin{equation}
\label{traj:limit}
\left\{
\begin{array}{l}
\ds{\frac{\dd\bY}{\dd t} \,=\,  \bF(t,\bY)   \,+\, g \,\frac{\nabla_\bx^\perp b}{b^2}(t,\bY) },
\\
\,
\\
\ds{\frac{\dd g}{\dd t} \,=\,  -\Div_\bx (\bF)(t,\bY)\,g},
\\
\,
\\
\bY(t^0) = \bx^0, \quad g(t^0)\,=\,e^0,
\end{array}\right.
\end{equation}
where $e^0\,=\,\frac{1}{2}\|\bv^0\|^2$.

\begin{remark}\label{rk:adiabatic}
Observe that $\thDiv_\bx(\bF)=-\bE^\perp\cdot\nabla_\bx(1/b)+\thRot(\bE)/b$. In particular, as expected, the limiting system \eqref{traj:limit} contains the fact that when $b$ depends only on time and $\bE$ derives from a potential, as in the Vlasov-Poisson case, the microscopic kinetic energy is conserved. Likewise, system~\eqref{traj:limit} also implies that 
$$
\frac{\dd}{\dd t}\left(\frac{g}{b(t,\bY)}\right) 
\,=\,-\frac{g}{b(t,\bY)^2}\left(\d_tb+\thRot(\bE)\right)(t,\bY)
$$
so that for the limiting system the classical adiabatic invariant $g/b$ becomes an exact invariant when the Maxwell-Faraday equation
$$
\d_tb+\thRot(\bE)\,=\,0
$$
holds. The latter occurs in particular when $b$ depends only on the space variable and $\bE$ derives from a potential, as in the Vlasov-Poisson case. For consistency's sake note also that when $b$ is constant the microscopic kinetic energy and the adiabatic invariant essentially coincide and their conservation is equivalent to the irrotationality of $\bE$.
\end{remark}

\subsection{Formal asymptotic limit of the Vlasov-Poisson system}

We come back to the Vlasov-Poisson system \eqref{eq:vlasov2d}. Here one cannot anymore remain completely at the characteristic level \eqref{traj:00}. Moreover whereas arguments of the previous subsection could be turned into sound analytic arguments, to the best of our knowledge the present situation does not fall directly into range of the actually available analysis of gyro-kinetic limits. We refer the reader to \cite{FS:00,GSR:99,SR:02,Cheverry,Miot,HerdaR} for a representative sample of such analytic techniques.

Nevertheless the previous subsection strongly suggests for $(f^\eps,\bE^\eps)$ solving the
Vlasov-Poisson system \eqref{eq:vlasov2d} that in the limit $\eps\rightarrow 0$, the electric field $\bE^\eps$ and the following velocity-averaged version of $f^\eps$ 
$$
f^\eps\,:\,(t,\bx,e)\mapsto\frac{1}{2\pi}\int_0^{2\pi} f^\eps(t,\bx, \sqrt{2e}(\cos(\theta),\sin(\theta)))\,\dd \theta
$$
converge to some $\bE$ and some\footnote{We use distinct notation of variables for limiting functions to be consistent with asymptotic analysis at the characteristic level. This is of course completely immaterial.} $f:(t,\by,g)\mapsto f(t,\by,g)$ solving the following system consisting in a transport equation supplemented with a Poisson equation, 
\begin{equation}
\left\{
\begin{array}{l}
 \ds \frac{\d f}{\d t}+\mathbf{U}\cdot\nabla_{\by} f+ u_g \frac{\d f}{\d g}=0,
\\ \, \\
\ds -\Delta_{\by}\phi=\rho\,,\quad
\rho=2\pi\,\int_{\RR^+} f \,\dd g,
\end{array}
\right.
   \label{eq:gc}
  \end{equation}
where the velocity field is given by
$$
{\bf U}(t,\by,g)\,=\,\bF(t,\by)  \,+\, g \frac{\nabla_\bx^\perp b}{b^2}(t,\by)\,, \qquad
u_g = -\Div_\by(\bF)(t,\by) \,g\,,
$$  
with $\bE = -\nabla_\by \phi$, $\bF=-\bE^\perp/b$. 

Our goal is not to develop a thorough analysis of system~\eqref{eq:gc} but let us mention that it generates a reasonable dynamics whose study falls into the scope of classical techniques.
\begin{theorem}\label{th:existence}
Consider $b\in W^{1,\infty}(\RR^+\times\RR^2)$ satisfying \eqref{hyp:1}.\\ 
Assume $f^0 \in L^1\cap L^\infty (\RR^2\times
\RR^+)$, $f^0$ is nonnegative and  
$$
\int_{\RR^2\times \RR^+}  g\,f^0(\by,g)\,\,\thdd\by\,\thdd g < \infty.
$$
Then system~\eqref{eq:gc} possesses a nonnegative weak solution $f$ starting from $f^0$ at time $0$ such that both Lebesgue norms of $f$ and total energy are non-increasing in time, in particular, 
$$
\|f(t,\cdot,\cdot) \|_{L^p} \,\leq\, \| f^0\|_{L^p}\, \quad \forall \,t\in \RR^+
$$
and
$$
\mathcal{E}(t) \,:=\, 2\pi\,\int_{\RR^2\times\RR^+} g\,f(t,\by, g)\,\thdd\by\,\thdd g
\,+\, \frac{1}{2}\int_{\RR^2} \|\bE(t,\by)\|^2 \thdd\by \,\leq \, \mathcal{E}(0)
\, \quad \forall \,t\in \RR^+.   
$$
Furthermore, when $f^0\in\mathcal{C}^1_c(\RR^2\times \RR^+)$, the solution is unique and preserved along the characteristic curves
$$
f(t,\bY(t),g(t)) \,=\, f(t^0, \bx^0, e^0) \quad \forall \,(t,t^0,\bx^0,e^0)\in \RR^+\times\RR^+\times\RR^2\times\RR^+
$$
where  $(\bY,g)$ solves \eqref{traj:limit}.
\end{theorem}

The foregoing result is not expected to be optimal in any reasonable way and its proof is completely analogous to those for the Vlasov-Poisson system (\cite{Arsenev,DiPernaLions} for weak solutions, \cite{Okabe} for smooth solutions). One key observation that underpins the analysis is that the vector field $(\bU,u_g)$ is divergence-free since
$$
\Div_\by \bU  \,+\, \frac{\d u_g}{\d g}  \,=\,\Div_\by(\bF)\,-\,\Div_\by(\bF) = 0.
$$
Note also that for smooth solutions both Lebesgue norms and total energy are exactly preserved by the time evolution. \bigskip

We now come to our main concern in the present article and seek after a numerical method that is able to capture these expected asymptotic properties, even when numerical discretization parameters are kept independent of $\eps$ hence are not adapted to the stiffness degree of the fast scales. Our objective enters in the general framework of so-called Asymptotic Preserving (AP) schemes, first introduced and widely studied for dissipative systems as in \cite{jin:99, klar:98}. Yet, in opposition with collisional kinetic equations in hydrodynamic or
diffusion asymptotics, collisionless equations like the Vlasov-Poisson
system \eqref{eq:vlasov2d} involve fast time oscillations instead of fast time relaxation. By many respects this makes the identification of suitable schemes much more challenging. 

In this context, among the few families of strategies available, most of them insist in providing also a good description of fast oscillations. To describe fast oscillations without imposing severe restrictions on time steps one needs in some sense to double time variables thus to see $f$ as the trace at $(t,\tau)=(t,t/\eps^2)$ of a function of variables $(t,\tau,\bx,\bv)$. The corresponding strategies have proved to be very successful when magnetic fields are uniform.

One of the oldest of these strategies, developed by E. Fr\'enod, F. Salvarani and E. Sonnendr\"ucker in \cite{FSS:09}, is directly inspired by theoretical results on two-scale convergence and relies on the fact that at the limit $\eps\to0$ the $\tau$-dependence may be explicitly filtered out. Its main drawback is probably that it computes only the leading order term in the limit $\eps\to0$. In particular it is only available when $\eps$ is very small. 

This may be fixed by keeping besides the stiff term to which a two-scale treatment is applied a non-stiff part that is smaller in the limit $\eps\to0$ but becomes important when $\eps$ is not small. Such a decomposition may be obtained by using a micro-macro
approach as in \cite{CFHM:15} and some references therein. This does allow to switch from one regime to another without any treatment of the transition between those but results in relatively heavy schemes,. 

Another approach with similar advantages, developed in \cite{bibCLM} and \cite{FHLS:15}, consists in explicitly doubling time variables and seeking higher-dimensional partial differential equations and boundary conditions in variables $(t,\tau,\bx,\bv)$ that contains the original system at the $\eps$-diagonal $(t,\tau)=(t,t/\eps)$. While the corresponding methods are extremely good at capturing oscillations their design require a deep \emph{a priori} understanding of the detailed structure of oscillations. 

Since there are dramatic changes in the structure of oscillations when magnetic fields are not uniform, it seems that the extension of any method capturing oscillations to realistic magnetic fields is either doomed to fail or at least to require a tremendous amount of work and complexity. Observe in particular that in inhomogeneous cases whereas at the limit $\eps\to0$ the slow evolution still obeys a closed system independent of fast scales (as \eqref{eq:gc} above), fast oscillations \emph{do not} uncouple anymore since their oscillating frequencies depend now on the slow scales... In any case this extension would constitute a major break-through in the field.\smallskip

Our goal is somehow more modest as we only require that, in the limit $\eps\to0$, our schemes capture accurately the non stiff part of the evolution while allowing for coarse discretization parameters. Recently, we have indeed proposed a class of semi-implicit schemes which allow to capture the asymptotic limit of the two dimensional Vlasov-Poisson system with a \emph{uniform} magnetic field \cite{FR16}. We stress that by many respects those schemes are remarkably natural and simple.

Here, we show how our approach may be extended to some non uniform cases hence allowing to obtain direct simulations of system \eqref{eq:vlasov2d} with time steps large with respect to $\eps$, while still capturing the slow coherent dynamics that emerge from strong oscillations. We develop numerical schemes that are able to deal with a wide range of values for $\varepsilon$ --- that is, AP schemes in the terminology mentioned above. Our schemes are consistent with the kinetic model for any positive value of $\eps$, and degenerate into schemes consistent with the asymptotic model~\eqref{eq:gc} when $\eps\rightarrow0$.

We stress that the extension carried out here is by no means a trivial continuation of \cite{FR16}. To understand inherent difficulties, we now review what are the main ingredients of the asymptotic analysis discussed in Section~\ref{s:external-asymptotics}, those that we aim at incorporating at the discrete level. As a preliminary warning, it must be stressed however that once the time variable has been discretized the distinction between strong and weak convergences essentially disappear. This is immaterial to the treatment of homogeneous cases since then, as implicitly contained in Remark~\ref{rk:homogeneous}, there is only one key-ingredient, the weak convergence of $\bV^\eps/\eps$ to $\bF(t,\bY)$, that may be safely replaced with a strong convergence at the discrete level. Here, beyond that requirement we also need three pieces of information essentially contained in Lemma~\ref{lmm:01} : weak convergences of $v_1^\eps\,v_2^\eps$ and $|v_1^\eps|^2-|v_2^\eps|^2$ to zero and strong convergence of $\frac12(|v_1^\eps|^2+|v_2^\eps|^2)$ to $g$ (solving \eqref{gc:2}). Therefore we need to offer compatible counterparts \emph{at the discrete level} of the weak convergence of $V^\eps$ to zero and the strong convergence of $\|V^\eps\|^2$ to a non trivial limit, a non trivial task !

\section{A particle method for inhomogeneous strongly magnetized
plasmas}
\setcounter{equation}{0}
\label{sec:3}

The schemes we shall propose here belong to the family of Particle-In-Cell (PIC) methods. Before presenting our specific schemes, we review in a few words the basic ideas underpinning these methods. We refer the reader to \cite{birdsall} for a thorough discussion and other applications to plasma physics. 

To keep notation as light as possible, we temporarily omit to denote the dependence of solutions on $\eps$ and discretization parameters. The starting point is the choice of approximating $f$, solving \eqref{eq:vlasov2d}, by a finite sum of smoothed Dirac masses. More explicitly one would like to compute
$$
f_N(t,\bx,\bv) \,:=\; \sum_{1\leq k\leq N } \omega_k \;\vp_\alpha (\bx-\bX_k(t)) \;\vp_\alpha (\bv-\bV_k(t))\,,
$$
where $\vp_\alpha = \alpha^{-d} \vp(\cdot / \alpha)$ is a particle shape function with radius proportional to $\alpha$ --- usually seen as a smooth approximation of the Dirac measure $\delta_0$ --- obtained by scaling a fixed compactly supported mollifier $\vp$, and the set $((\bX_k,\bV_k))_{1\leq k\leq N}$ represents the position in phase space
of $N$ macro-particles evolving along characteristic curves \eqref{traj:00} from initial data $(\bx_k^0, \bv_k^0)$, $1\leq k \leq N$. More explicitly 
\begin{equation}
\label{traj:bis}
\left\{
\begin{array}{l}
\ds\eps\frac{\dd\bX_k}{\dd t} \,=\, \bV_k, 
\\
\,
\\
\ds\eps\frac{\dd\bV_k}{\dd t} \,=\, -\frac{1}{\eps}\,b(t,\bX_k)\,
\bV_k^\perp  \,+\, \bE(t,\bX_k), 
\\
\,
\\
\bX_k(0) = \bx^0_k, \quad\bV_k(0) = \bv^0_k,
\end{array}\right.
\end{equation}
where the electric field $\bE$ solves the Poisson equation \eqref{eq:vlasov2d}, or, more exactly in the end, it is computed from a solution of a discretization of the Poisson equation on a mesh of the physical space. Parameters are initially chosen to ensure that $f_N(0,\bx,\bv)$ is a good approximation of (continuous) initial datum $f^0$. To increase the order of approximation of Dirac masses, mollifiers are sometimes chosen with a prescribed number of vanishing moments. Concrete common choices include B-splines. See for instance \cite{Koumoutsakos.1997.jcp,Cottet.Koumoutsakos.2000.cup}.

The main remaining issue is then to design an accurate and stable approximation of the particle trajectories in phase space. Note moreover that we ultimately want those properties to be uniform with respect to $\eps$ at least for the slow part of the evolution. As already mentioned, to achieve this goal, we need to reproduce at the discrete level both the (weak) convergence of each $\bV_k$ to zero and a non trivial slow dynamics of its microscopic energy  $e_k=\tfrac{1}{2}\|\bV_k\|^2$. We resolve this difficulty by augmenting the dimension of the phase space by one, looking for unknowns $(\bX_k,\bw_k,e_k)$ starting from $(\bx^0_k,\bv^0_k,\tfrac12\|\bv_k^0\|^2)$ and providing a solution to \eqref{traj:bis} by setting $(\bX_k,\bV_k)=(\bX_k,\sqrt{2e_k}\ \bw_k/\|\bw_k\|)$. The simplest possible choice of such systems would be
\begin{equation}
\left\{
\begin{array}{l}
\ds\eps\frac{\dd\bX_k}{\dd t} \,=\, \bw_k, 
\\
\,
\\
\ds\eps\frac{\dd e_k}{\dd t} \,=\, \bE(t,\bX_k)\cdot\bw_k,
\\
\,
\\
\ds\eps\frac{\dd\bw_k}{\dd t} \,=\, -\frac{1}{\eps}\,b(t,\bX_k)\,
\bw_k^\perp  \,+\, \bE(t,\bX_k), 
\\
\,
\\
\bX_k(0) = \bx^0_k, \quad e_k(0)\,=\, e_k^0,\quad \bw_k(0) = \bw^0_k.
\end{array}\right.
\label{traj:ter}
\end{equation}
with $(\bw_k^0,e_k^0)$ starting from $(\bv^0_k,\tfrac12\|\bv_k^0\|^2)$. However this choice is far from being uniquely determined. For instance one may add to the vector field defining System~\eqref{traj:ter} any function of $(t,\bX,e,\bw)$ that vanishes when $e=\tfrac12\|\bw\|^2$. In the following we shall rather work with
\begin{equation}
\left\{
\begin{array}{l}
\ds\eps\frac{\dd\bX_k}{\dd t} \,=\, \bw_k, 
\\
\,
\\
\ds\eps\frac{\dd e_k}{\dd t} \,=\, \bE(t,\bX_k)\cdot\bw_k,
\\
\,
\\
\ds\eps\frac{\dd\bw_k}{\dd t} \,=\, -\frac{1}{\eps}\,b(t,\bX_k)\,
\bw_k^\perp  \,+\, \bE(t,\bX_k)\,-\,\chi(e_k,\bw_k)\,\nabla_\bx (\ln(b))(t,\bX_k), 
\\
\,
\\
\bX_k(0) = \bx^0_k, \quad e_k(0)\,=\, e_k^0,\quad \bw_k(0) = \bw^0_k
\end{array}\right.
\label{traj:qua}
\end{equation}
with $\chi$ chosen such that
\begin{equation}
\label{hyp:2}
\left(\chi\left(\tfrac{1}{2}\|\bw\|^2\,,\, \bw\right) =0,\,\, \forall \bw\in\RR^2\right)\quad{\rm and}
\quad \left(\lim_{\bw\rightarrow 0} \chi(e,\bw) = e, \,\, \forall e\in\RR\right)
\end{equation} 
and
\begin{equation}
\label{hyp:2bis}
0 \leq \chi(e,\bw) \leq e, \quad \forall(e,\bw)\in\RR_+\times\RR^2\,.
\end{equation}
For concreteness, in the following, we actually choose $\chi$ as
$$
\chi(e,\bw) \,=\, \frac{e}{e \,+\, \|\bw\|^2/2 }\,\left( e \,-\, \frac{\|\bw\|^2}{2} \right)^+\,,\quad \forall(e,\bw)\in\RR\times\RR^2
$$ 
where $s^+=\max(0,s)$. The choice of System~\ref{traj:qua} originates in the fact that it is relatively easy to discretize it so as to ensure that, in the limit $\eps\to0$, $\eps^{-1}\bw_k$ becomes asymptotically close to
$$
\bF(t,\bX_k)\,-\,e_k\nabla_\bx^\perp \left(\frac1b\right)(t,\bX_k)
$$
(as it occurs for $\eps^{-1}\bV_k$, see Section~\ref{s:external-asymptotics}). Since we start with $e_k^0=\tfrac12\|\bw_k^0\|^2$ and the constraint is preserved by the evolution, the choice of either System~\eqref{traj:ter} or of \eqref{traj:qua} or of any similar system is perfectly immaterial as long as we do not discretize the time variable. However their discretizations will no longer preserve exactly the condition $e_k^0=\tfrac12\|\bw_k^0\|^2$ thus they will lead to effectively distinct solutions.

Note that the addition of a new phase space variable is much less invasive than doubling time variables (but is insufficient to allow for the description of fast oscillations with coarse grids) and it is also very cheap in the context of PIC methods. In particular, consistently with the former claim, now that once we have chosen an augmented system such as~\eqref{traj:qua} we may discretize it by adapting almost readily the strategy developed in \cite{BFR:15, FR16} and based on semi-implicit solvers for stiff problems. 

In the rest of this section, we shall focus on the discretization of System~\eqref{traj:qua} and propose several numerical schemes for that purpose. Yet, prior to that, we briefly summarize the above discussion and explain how this discretization will be used to compute solutions to \eqref{eq:vlasov2d}. Besides already introduced parameters, fix a given time step $\Delta t>0$ and for $n\geq0$ introduce discrete time $t^n=n\,\Delta t$. Then we chose a discretization of System~\eqref{traj:qua} expected to provide for any $k\in\{1,\ldots,N\}$, $(\bx^n_k,,e^n_k,\bw^n_k)$ an approximation of $(\bX_k(t^n),e_k(t^n),\bw_k(t^n))$, where $(\bX_k,e_k,\bw_k)$ solves \eqref{traj:qua}. Of course equations are not uncoupled thus we compute simultaneously all $(\bx^n_k,,e^n_k,\bw^n_k)$, $k\in\{1,\ldots,N\}$, and the corresponding electric field computed from a solution on a mesh of the physical space of a discretization of the Poisson equation  with macroscopic density corresponding to the distribution function $(f_{N,\alpha}^n)_{n\geq0}$, given at time $t^{n}$ by
$$
f_{N,\alpha}^{n}(\bx,\bv) \,:=\; \sum_{1\leq k \leq N} \omega_k
\;\vp_\alpha (\bx-\bx^n_k) \;\vp_\alpha (\bv-\bv^n_k)\,,
$$
where
$$
\bv^n_k \,\;:=\,\, \sqrt{2 \, e_k^n} \,\, \frac{\bw^n_k}{\|\bw_k^n\|}\,. 
$$

\subsection{A first-order semi-implicit scheme}
Since we only discuss possible discretizations of \eqref{traj:qua}, we shall omit from now on the index $k\in\{1,\ldots,N\}$ and consider the electric field as fixed, assuming smoothness of $(\bE,b)$ and  
\eqref{hyp:1}.

We start with the simplest semi-implicit scheme for \eqref{traj:qua}, which is a combination of the backward and forward Euler schemes. For a fixed time step $\Delta t>0$ it is given by
\begin{equation}
\label{scheme:0}
\left\{
\begin{array}{l}
\ds\frac{\bx^{n+1} - \bx^n }{\Delta t} \,\,=\, \frac{\bw^{n+1}}{\eps}\,,
\\
\,
\\
\ds\frac{e^{n+1} - e^n }{\Delta t} \,\,=\, \bE(t^n,\bx^n)\cdot \frac{\bw^{n+1}}{\eps}\,,
\\
\,
\\
\ds\frac{\bw^{n+1} - \bw^n }{\Delta t} \,=\, \frac{1}{\eps}\left(
{\bE(t^n,\bx^n)} \,-\, \chi(e^n,\bw^n)\,\nabla_\bx(\ln(b))(t^n,\bx^n) - b(t^n,\bx^n)\frac{(\bw^{n+1})^\perp }{\eps} \right).
\end{array}\right.
\end{equation}
Notice that only the third equation on $\bw^{n+1}$ is really implicit and it only requires the resolution of a two-dimensional linear system. Then, once the value of $\bw^{n+1}$ has been computed the first and second equations provide explicitly the values of $\bx^{n+1}$ and $e^{n+1}$.

\begin{proposition}[Consistency in the limit $\eps\rightarrow 0$ for a fixed $\Delta t$]
\label{prop:1}
Let us consider a time step $\Delta t>0$, a final time $T>0$  and
set $N_T=\lfloor T/\Delta t\rfloor$. Assume that $(\bx^n_\eps,\bw_\eps^n,e^n_\eps)_{0\leq
  n\leq N_T}$ is a sequence obtained by \eqref{scheme:0} and such that
\begin{itemize}
\item for all $1\leq n\leq N_T$, $\left(\bx^n_\eps,\eps \bw^n_\eps, e^n_\eps\right)_{\eps>0}$ is uniformly bounded with respect to $\eps>0$;
\item $\left(\bx^0_\eps,\bw_\eps^0,e^0_\eps\right)_{\eps>0}$ converges in the limit $\eps\rightarrow 0$ to some $(\by^0,\bw^0,g^0)$. 
\end{itemize}
Then, for any $1\leq n\leq N_T$, $(\bx^n_\eps,e^n_\eps)_{\eps>0}$ converges to some $(\by^n,g^n)$ as $\eps\rightarrow 0$ and the limiting sequence $(\by^n,g^n)_{1\leq n\leq N_T}$ solves
\begin{equation}
\label{sch:y0}
\left\{
\begin{array}{l}
\ds\frac{\by^{n+1} - \by^n}{\Delta t} \,=\, -\frac{1}{b(t^n,\by^n)} \Big(
  \bE(t^n,\by^n) - g^n \,\nabla_\by (\ln(b))(t^n,\by^n)\Big)^\perp
\\
\,
\\
\ds\frac{g^{n+1} - g^n}{\Delta t} = g^n \,\frac{\nabla_\by^\perp b}{b^2}(t^n,\by^n) \cdot \bE(t^n,\by^n)\,
\end{array}\right.
\end{equation}
which provides a consistent first-order approximation with respect to $\Delta t$ of the
gyro-kinetic system~\eqref{traj:limit}.
\end{proposition}

Note that we actually apply this result with $\left(\bx^0_\eps,\bw_\eps^0,e^0_\eps\right)_{\eps>0}=(\bx^0,\bv^0,\tfrac12\|\bv^0\|^2)$ so that in this case $(\by^0,\bw^0,g^0)=(\bx^0,\bv^0,\tfrac12\|\bv^0\|^2)$. Furthermore, in this case, the limiting scheme \eqref{sch:y0} starts from initial data
$$
\by^1\,=\,\by^0\,+\,(\Delta t)\,\bF(0,\by^0)\,,\qquad\qquad
g^1\,=\,g^0\,.
$$
More generally, under the assumptions of Proposition~\ref{prop:1}, $(\by^1,g^1)$ differs from  $(\by^0,g^0)$ by a first-order term but in general does not satisfy \eqref{sch:y0} with $n=0$. The scheme needs one time-step to reach the expected asymptotic regime in the limit $\eps\to0$. 

\begin{proof}
First, note that, for any $0\leq n\leq N_T$, since $(\bx^n_\eps,e^n_\eps)_{\eps>0}$ is bounded, $\bE(t^n,\bx^n_\eps)$, $\nabla_\bx b(t^n,\bx^n_\eps)$ and $\chi(e^n_\eps,\bw_\eps^n)$ are also bounded uniformly in $\eps$. Therefore, combined with assumption \eqref{hyp:1} and boundedness of $\left(\eps \bw^n_\eps\right)_{\eps>0}$, the third equation of \eqref{scheme:0}
written as
\begin{equation}
\label{flav:0}
\frac{(\bw^{n+1}_\eps)^\perp}{\eps}\,\,=\,\frac{1}{b(t^n,\bx_\eps^n)}\left(-\frac{\eps\bw^{n+1}_\eps - \eps\bw^n_\eps }{\Delta t}\,+\, 
\bE(t^n,\bx^n_\eps) \,-\, \chi(e^n_\eps,\bw_\eps^n)\,\nabla_\bx \ln(b(t^n,\bx^n))\right)
\end{equation}
shows that, for any $1\leq n\leq N_T$, $\left(\eps^{-1}\bw^n_\eps\right)_{\eps>0}$ is bounded and in particular $\left(\bw^n_\eps\right)_{\eps>0}$ converges to zero. 
Moreover, for any $1\leq n\leq N_T$, since the sequence $(\bx^n_\eps, e^n_\eps)_{\eps>0}$ is uniformly bounded with respect to
$\eps>0$, we can extract a subsequence still abusively labeled by $\eps$ and find some 
$(\by^n, g^n)$ such that $(\bx^n_\eps, e^n_\eps) \rightarrow (\by^n,g^n)$ as $\eps$ goes to zero.

Therefore, by using continuity of $(E,b,\nabla_\bx b,\chi)$, we first conclude that for any $1\leq n\leq N_T-1$, \eqref{flav:0} yields
$$
\eps^{-1}\,\bw^{n+1}_\eps \,\rightarrow\,
-\frac{1}{b(t^n,\by^n)}\Big( \,\bE(t^n,\by^n) \,-\, g^n\,\nabla_\by
  (\ln(b))(t^n,\by^n) \,\Big)^\perp, \,\,{\rm when}\,\, \eps\rightarrow 0
$$
since $\chi(g^n,0)=g^n$. By substituting the foregoing limit in the first and second  equations of \eqref{scheme:0} we obtain that the limit $(\by^n,g^n)_{1\leq n\leq N_T}$ satisfies \eqref{sch:y0}. Moreover the same argument shows that
\begin{equation}\label{sch:y0-init}
\left\{
\begin{array}{l}
\ds\by^1=\,\by^0\,+\,(\Delta t)\,\left(\bF(0,\by^0)
\,+\,\chi(g^0,\bw^0)\,\frac{\nabla_\by^\perp b}{b^2}(0,\by^0)\right)\,,\\
\ds g^1=\,g^0\,+\,(\Delta t)\,\chi(g^0,\bw^0)\,\bE(0,\by^0)\cdot\frac{\nabla_\by^\perp b}{b^2}(0,\by^0)\,.
\end{array}\right.
\end{equation}
Since the limit points $(\by^n,g^n)$ of the extracted subsequence are uniquely determined (by \eqref{sch:y0} and \eqref{sch:y0-init}), actually all the sequence $(\bx^n_\eps,e^n_\eps)_{\eps>0}$ converges (for any $n$).
\end{proof}

\begin{remark}
The consistency provided by the latter result is far from being uniform with respect to the time step. However, though we restrain from doing so here in order to keep technicalities to a bare minimum, we expect that an analysis similar to the one carried out in \cite[Section~4]{FR16} could lead to uniform estimates, proving uniform stability and  consistency with respect to both $\Delta t$ and $\eps$. 
\end{remark}

Of course, a first-order scheme may fail to be accurate enough to describe correctly the long time behavior of the solution, but it has the advantage of simplicity. In the following we show how to generalize our approach to second and third-order schemes. 

\subsection{Second-order semi-implicit Runge-Kutta schemes}
Now, we come to second-order schemes with two stages.  The scheme we consider is a combination of a Runge-Kutta method for the explicit part and of an $L$-stable second-order SDIRK method for the implicit part. 

\begin{remark}
In our original treatment of the homogeneous case in \cite{FR16} where we design schemes to capture \eqref{gc:-1}, the discrete velocity $\bv^n_\eps$ was damped to zero in the limit $\eps\to0$, consistently with weak convergence of $\bV^\eps$ to zero but not with conservation of $\tfrac12\|\bV^\eps\|^2$. Therefore there was some interest there in tailoring schemes not too dissipative. Here at the discrete level the two distinct behaviors are encoded by two distinct variables $\bw^n_\eps$ and $e^n_\eps$ so that we may rightfully focus solely on implicit $L$-stable choices.
\end{remark}

To describe the scheme, we introduce $\gamma>0$ the smallest root of the polynomial $X^2 - 2X + 1/2$, {\it  i.e.} $\gamma = 1 - 1/\sqrt{2}$. Then the scheme is given by the
following two stages. First,
\begin{equation}
\label{scheme:3-1}
\left\{
\begin{array}{l}
\ds \bx^{(1)} \,=\, \bx^n  \,+\, \frac{\gamma\Delta t}{\eps}\,\bw^{(1)}\,,
\\
\,
\\
\ds e^{(1)} \,=\, e^n  \,+\, \frac{\gamma\Delta t}{\eps} \,S^{(1)}\,,
\\
\,
\\
\ds{\bw^{(1)} \,=\, \bw^n  \,+\, \frac{\gamma\Delta t}{\eps}\,\bF^{(1)},}
\end{array}\right.
\end{equation}
with
\begin{equation}
\label{F1}
\left\{
\begin{array}{l}
\ds \bF^{(1)} \,:=\, \bE(t^n,\bx^n) \,-\,
\chi(e^n,\bw^n)\,\nabla_\bx (\ln(b))(t^n,\bx^n) \,-\,b(t^n,\bx^n)\,\frac{(\bw^{(1)})^\perp}{\eps}\,,\\[0.5em]
S^{(1)}\,:=\, \bE(t^n,\bx^n)\cdot \bw^{(1)}\,.
\end{array}\right.
\end{equation}
Before the second stage, we first introduce  $\hat{t}^{(1)} \,=\, t^n + {\Delta t}/{(2\gamma)}$ and explicitly compute $(\hat{\bx}^{(1)},\hat{\bw}^{(1)},\hat{e}^{(1)})$ from
\begin{equation}
\label{tard2}
\left\{
\begin{array}{l}
\ds\hat{\bx}^{(1)} \,=\, \bx^{n}\,+\, \frac{\Delta t}{2\gamma\eps}
\bw^{(1)}
\,=\, \bx^{n}\,+\, \frac{\bx^{(1)}-\bx^{n}}{2\gamma^2},\\ \, \\ 
\ds\hat{e}^{(1)} \,=\, e^{n}\,+\, \frac{\Delta t}{2\gamma\eps}
\,S^{(1)}\,=\, e^{n}\,+\, \frac{e^{(1)}-e^{n}}{2\gamma^2}, 
\\ \,\\
\ds \hat{\bw}^{(1)} \,=\,
\bw^{n}\,+\, \frac{\Delta t}{2\gamma\eps} \bF^{(1)}\,=\,
\bw^{n}\,+\, \frac{\bw^{(1)}-\bw^{n}}{2\gamma^2}\,.
\end{array}\right.
\end{equation}
Then  the solution of the second stage $(\bx^{(2)},\bw^{(2)},e^{(2)})$ is given by
\begin{equation}
\label{scheme:3-2}
\left\{
\begin{array}{l}
\ds{\bx^{(2)} \,=\, \bx^{n}  \,+\, \frac{(1-\gamma)\Delta t}{\eps} \,\bw^{(1)}  \,+\, \frac{\gamma\Delta t}{\eps}\,\bw^{(2)},}
\\
\,
\\
\ds{e^{(2)} \,=\, e^{n}  \,+\, \frac{(1-\gamma)\Delta t}{\eps} \,S^{(1)} \,+\, \frac{\gamma\Delta t}{\eps}\,S^{(2)},}
\\
\,
\\
\ds{\bw^{(2)} \,=\, \bw^{n}  \,+\, \frac{(1-\gamma)\Delta t}{\eps}
  \,\bF^{(1)}  \,+\, \frac{\gamma\Delta t}{\eps}  \,\bF^{(2)}},
 \end{array}\right.
\end{equation}
with 
\begin{equation}
\label{F2}
\left\{
\begin{array}{l}
\ds \bF^{(2)} :=\bE\left(\hat{t}^{(1)}, \hat{\bx}^{(1)}\right) -
\chi\left(\hat{e}^{(1)},\hat{\bw}^{(1)}\right)\nabla_\bx(\ln\left(b\right))\left(\hat{t}^{(1)}, \hat{\bx}^{(1)}\right)
-b\left(\hat{t}^{(1)}, \hat{\bx}^{(1)}\right) \frac{(\bw^{(2)})^\perp}{\eps}\,,
\\
\;
\\
S^{(2)} := \bE(\hat{t}^{(1)},\hat{\bx}^{(1)})\cdot \bw^{(2)}\,.
\end{array}\right.
\end{equation}
Finally, the numerical solution at the next time step is defined by
\begin{equation}
\label{scheme:3-3}
\bx^{n+1} \,=\, \bx^{(2)},\qquad
e^{n+1} \,=\, e^{(2)},\qquad
\bw^{n+1} \,=\, \bw^{(2)}.
\end{equation}

\begin{proposition}[Consistency in the
  limit $\eps\rightarrow 0$ for a fixed $\Delta t$]
\label{prop:3}
Let us consider a time step $\Delta t>0$, a final time $T>0$  and
set $N_T=\lfloor T/\Delta t\rfloor$.
Assume that $(\bx^n_\eps,\bw_\eps^n,e^n_\eps)_{0\leq
  n\leq N_T}$ is a sequence obtained by \eqref{scheme:3-1}-\eqref{scheme:3-3} and such that
\begin{itemize}
\item for all $1\leq n\leq N_T$, $\left(\bx^n_\eps,\eps \bw^n_\eps, e^n_\eps\right)_{\eps>0}$ is uniformly bounded with respect to $\eps>0$;
\item $\left(\bx^0_\eps,\bw_\eps^0,e^0_\eps\right)_{\eps>0}$ converges in the limit $\eps\rightarrow 0$ to some $(\by^0,\bw^0,g^0)$. 
\end{itemize}
Then, for any $1\leq n\leq N_T$, $(\bx^n_\eps,e^n_\eps)_{\eps>0}$ converges to some $(\by^n,g^n)$ as $\eps\rightarrow 0$ and the limiting sequence $(\by^n,g^n)_{0\leq n\leq N_T}$ solves 
\begin{equation}
\label{sch:y2-bis}
\left\{
\begin{array}{l}
\ds \by^{n+1} = \ds \by^{n} + (1-\gamma)\,\Delta t \,\bU^{n} +
  \gamma\,\Delta t\,\,\bU^{(1)}\,,
\\
\,
\\
\ds g^{n+1}  = \ds g^{n} + (1-\gamma)\,\Delta t \, \,T^n \,+\, \gamma\,\Delta t\,T^{(1)} \,,
\end{array}\right.
\end{equation}
where
$$
\left\{
\begin{array}{l}
\ds \bU^{n} = -\frac{1}{b(t^n,\by^n)} \Big( \bE(t^n,\by^n) - g^n \,\nabla_\by (\ln(b))(t^n,\by^n)\Big)^\perp,
\\
\,
\\
\ds T^{n} = g^n \,\frac{\nabla_\by^\perp b}{b^2}(t^n,\by^n) \cdot \bE(t^n,\by^n)
\end{array}\right.
$$
and
$$
\left\{
\begin{array}{l}
\ds \bU^{(1)} = \ds -\frac{1}{b\left(\hat{t}^{(1)},\hat{\by}^{(1)}\right)} \Big(
  \bE\left(\hat{t}^{(1)},\hat{\by}^{(1)}\right) - \hat{g}^{(1)}
  \,\nabla_\by(\ln\left(b\right))\left(\hat{t}^{(1)},\hat{\by}^{(1)}\right)
  \Big)^\perp,
\\
\,
\\
\ds T^{(1)} = \hat{g}^{(1)}\,\frac{\nabla_\by^\perp b}{b^2}
\left(\hat{t}^{(1)},\hat{\by}^{(1)}\right) \cdot \bE\left(\hat{t}^{(1)},\hat{\by}^{(1)}\right)\,,
\end{array}\right.
$$
with $\hat{t}^{(1)}=t^n+\Delta t/(2\gamma)$ and
$$
\left\{
\begin{array}{l}
\ds \hat{\by}^{(1)} = \by^{n} + \frac{\Delta t}{2\gamma} \,\bU^{n}\,,
\\
\,
\\
\ds \hat{g}^{(1)} = g^n \,+\,  \frac{\Delta t}{2\gamma}  \,T^n\,,
\end{array}\right.
$$
which provides a consistent second-order approximation of the gyro-kinetic system \eqref{traj:limit}.
\end{proposition}

\begin{proof}
We mainly follow the lines of the proof of Proposition~\ref{prop:1}. To make arguments more precise we mark with a suffix ${}_{n+1,\eps}$ intermediate quantities involved in the step from $t^n$ to $t^{n+1}$.

To begin with, for $1\leq n\leq N_T$, the third equation of \eqref{scheme:3-1} implies that $(\eps^{-1} \bw^{(1)}_{n,\eps})_\eps$ is bounded, then the first and second equations of \eqref{tard2} show that $(\hat{\bx}^{(1)}_{n,\eps},\hat{e}^{(1)}_{n,\eps})_\eps$ is also bounded, from which the third equation of \eqref{scheme:3-1} shows the boundedness of $(\eps^{-1} \bw^{(1)}_{n,\eps})_\eps$. Then from the third equation of \eqref{tard2} we derive that $(\hat{\bw}^{1}_{1,\eps})_\eps$ is bounded and that, for $2\leq n\leq N_T$, $(\eps^{-1}\hat{\bw}^{(2)}_{n,\eps})_\eps$ is bounded.

One may then extract converging subsequences from all bounded families. With obvious notation for limits this leads in particular along the chosen subsequence from the third equations of \eqref{scheme:3-1} and \eqref{scheme:3-2} to
$$
\eps^{-1}\,\bw^{(1)}_{n+1,\eps} \,\stackrel{\eps\to0}{\longrightarrow}\, 
-\frac{1}{b(t^n,\by^n)}\Big( \,\bE(t^n,\by^n) \,-\, g^n\,\nabla_\by
  (\ln(b))(t^n,\by^n) \,\Big)^\perp
$$
and
$$
\eps^{-1}\,\bw^{(2)}_{n+1,\eps}\,\stackrel{\eps\to0}{\longrightarrow}\, -\frac{1}{b(t^{(1)}_{n+1},\hat{\by}^{(1)}_{n+1})}\Big( \,\bE\left(\hat{t}^{(1)}_{n+1},\hat{\by}^{(1)}_{n+1}\right) \,-\, \hat{g}^{(1)}_{n+1}\,\nabla_\by(\ln\left(b\right))(\hat{t}^{(1)}_{n+1},\hat{\by}^{(1)}_{n+1})\,\Big)^\perp 
$$
when $1\leq n\leq N_T-1$, since $\bw^n_\eps\to0$ and $\bw^{(1)}_{n+1,\eps}\to0$. With this in hands taking, along the same subsequence, the limit $\eps\to0$ of first and second equations of \eqref{tard2} and \eqref{scheme:3-2} provides the claimed scheme.

Moreover the same arguments provide explicit expressions of $(\by^1,g^1)$ in terms of $(\by^0,\bw^0,g^0,\Delta t)$. This yields uniqueness of limits independently of the chosen subsequence hence convergence of full families.
\end{proof}

\subsection{Third-order semi-implicit Runge-Kutta schemes}
Now we consider a third-order semi-implicit scheme. It is given by a four stages Runge-Kutta method and was introduced in the framework of hyperbolic systems with stiff source terms in \cite{BFR:15}. 

First, we choose $\alpha=0.24169426078821$, $\eta= 0.12915286960590$ and set $\beta =
\alpha/4$, $\gamma=1/2-\alpha-\beta-\eta$. Then the first stage of the scheme is
\begin{equation}
\label{scheme:4-1}
\left\{
\begin{array}{l}
\ds{\bx^{(1)} \,=\, \bx^n  \,+\, \frac{\alpha\Delta t}{\eps}\,\bw^{(1)},}
\\
\,
\\
\ds{e^{(1)} \,=\, e^n  \,+\, \frac{\alpha\Delta t}{\eps}\, S^{(1)},}
\\
\,
\\
\ds{\bw^{(1)} \,=\, \bw^n  \,+\, \frac{\alpha\Delta t}{\eps} \,\bF^{(1)},}
\end{array}\right.
\end{equation}
with 
\begin{equation}
\label{3F1}
\left\{
\begin{array}{l}
\ds \bF^{(1)} \,:=\, \bE(t^n,\bx^n) \,-\, \chi(e^n,\bw^n)\,\nabla_\bx (\ln(b))(t^n,\bx^n)
\,-\,b(t^n,\bx^n)\,\frac{(\bw^{(1)})^\perp}{\eps}\,,
\\[0.5em]
S^{(1)}\,:=\;\bE(t^n,\bx^n)\cdot\bw^{(1)}\,.
\end{array}\right.
\end{equation}
The second stage is
\begin{equation}
\label{scheme:4-2}
\left\{
\begin{array}{l}
\ds{\bx^{(2)} \,=\, \bx^{n}  \,-\, \frac{\alpha\Delta t}{\eps} \,\bw^{(1)}  \,+\, \frac{\alpha\Delta t}{\eps}\,\bw^{(2)}\,,}
\\
\,
\\
\ds{e^{(2)} \,=\, e^{n}  \,-\, \frac{\alpha\Delta t}{\eps} \,S^{(1)}
  \,+\, \frac{\alpha\Delta t}{\eps}\, S^{(2)}\,,}
\\
\,
\\
\ds{\bw^{(2)} \,=\, \bw^{n}  \,-\, \frac{\alpha\Delta t}{\eps} \,\bF^{(1)} \,+\, \frac{\alpha\Delta t}{\eps}\,\bF^{(2)}\,,}
\end{array}\right.
\end{equation}
with
\begin{equation}
\label{3F2}
\left\{
\begin{array}{l}
\ds \bF^{(2)} \,:=\, \bE(t^n,\bx^n) \,-\,
\chi(e^n,\bw^n)\,\nabla_\bx (\ln(b))(t^n,\bx^n)
\,-\,b(t^n,\bx^n)\,\frac{(\bw^{(2)})^\perp}{\eps}\,,
\\
\,
\\
S^{(2)}\,:=\,\bE(t^n,\bx^n)\cdot\bw^{(2)}.
\end{array}\right.
\end{equation}
Then, for the third stage we first compute explicitly intermediate values
$$
\left\{
\begin{array}{l}
\ds{\hat{\bx}^{(2)} \,:=\, \bx^{n} + \frac{\Delta t}{\eps}\,\bw^{(2)}
\,=\,\bx^n+\frac1\alpha\left(\bx^{(1)}-\bx^n+\bx^{(2)}-\bx^n\right),}
\\
\,
\\
\ds{\hat{e}^{(2)} \,:=\, e^{n} + \frac{\Delta t}{\eps}\,S^{(2)}
\,=\,e^n+\frac1\alpha\left(e^{(1)}-e^n+e^{(2)}-e^n\right)\,,}
\\
\,
\\
\ds{\hat{\bw}^{(2)} \,:=\, \bw^{n} + \frac{\Delta t}{\eps}\,\bF^{(2)}
\,=\,\bw^n+\frac1\alpha\left(\bw^{(1)}-\bw^n+\bw^{(2)}-\bw^n\right)}
\end{array}\right.
$$
and then construct the new stage
$\left(\bx^{(3)},e^{(3)},\bw^{(3)}\right)$ 
\begin{equation}
\label{scheme:4-3}
\left\{
\begin{array}{l}
\ds{\bx^{(3)} \,=\,   \bx^{n}  \,+\, \frac{(1-\alpha)\Delta t}{\eps}\,\bw^{(2)} \,+\, \frac{\alpha\Delta t}{\eps}\,\bw^{(3)},}
\\
\,
\\
\ds{e^{(3)} \,=\,   e^{n}  \,+\, \frac{(1-\alpha)\Delta t}{\eps}\,S^{(2)} \,+\, \frac{\alpha\Delta t}{\eps}\,S^{(3)},}
\\
\,
\\
\ds{\bw^{(3)} \,=\, \bw^{n} \,+\,  \frac{(1-\alpha)\Delta
    t}{\eps}\,\bF^{(2)} \,+\, \frac{\alpha\Delta
  t}{\eps}\,\bF^{(3)},}
\end{array}\right.
\end{equation}
with
$$
\left\{
\begin{array}{l}
\ds{\bF^{(3)} \,:=\, \bE\left(t^{n+1},\hat{\bx}^{(2)}\right) \,-\, \chi\left(\hat{e}^{(2)},\hat{\bw}^{(2)}\right)\,\nabla_\bx (\ln\left(b\right))(t^{n+1},\hat{\bx}^{(2)}) \,-\,b\left(t^{n+1},\hat{\bx}^{(2)}\right)\,\frac{(\bw^{(3)})^\perp}{\eps}\,,}
\\
\,
\\
\ds{S^{(3)} \,:=\, \bE\left(t^{n+1},\hat{\bx}^{(2)}\right)\cdot \bw^{(3)}\,.}
\end{array}\right.
$$
Finally, we use explicit intermediate values
$\left(\hat{\bx}^{(3)}, \hat{e}^{(3)},\hat{\bw}^{(3)}\right)$ 
$$
\left\{
\begin{array}{l}
\ds \hat{\bx}^{(3)} \,:=\, \bx^{n}  \,+\, \frac{\Delta t}{4\eps}
\left( \bw^{(2)}  + \bw^{(3)} \right)
\,=\,\bx^n+\frac{1}{4\alpha}\left(\bx^{(3)}-\bx^n-\frac{1-2\alpha}{\alpha}(\bx^{(2)}-\bx^n+\bx^{(1)}-\bx^n)\right),
\\
\,
\\
\ds \hat{e}^{(3)} \,:=\, e^{n}  \,+\, \frac{\Delta t}{4\eps}
\left( S^{(2)} + S^{(3)} \right)
\,=\,e^n+\frac{1}{4\alpha}\left(e^{(3)}-e^n-\frac{1-2\alpha}{\alpha}(e^{(2)}-e^n+e^{(1)}-e^n)\right),
\\
\,
\\
\ds \hat{\bw}^{(3)} \,:=\, \bw^{n}  \,+\, \frac{\Delta t}{4\eps}
\left( \bF^{(2)}  + \bF^{(3)} \right)
\,=\,\bw^n+\frac{1}{4\alpha}\left(\bw^{(3)}-\bw^n-\frac{1-2\alpha}{\alpha}(\bw^{(2)}-\bw^n+\bw^{(1)}-\bw^n)\right),
\end{array}\right.
$$
to carry out the fourth stage 
\begin{equation}
\label{scheme:4-4}
\left\{
\begin{array}{l}
\ds\bx^{(4)}  \,=\,  \bx^{n} \,+\, \frac{\beta\Delta
  t}{\eps}\,\bw^{(1)}\,+\, \frac{\eta\Delta t}{\eps}\,\bw^{(2)} \,+\, \frac{\gamma\Delta t}{\eps}\,\bw^{(3)}\,+\, \frac{\alpha\Delta t}{\eps}\,\bw^{(4)}\,,
\\
\,
\\
\ds e^{(4)}  \,=\,  e^{n} \,+\, \frac{\beta\Delta
    t}{\eps}\,S^{(1)}\,+\,\frac{\eta\Delta  t}{\eps}\,S^{(2)}\, +\, \frac{\gamma\Delta t}{\eps}\,S^{(3)}\,+\, \frac{\alpha\Delta t}{\eps}\,S^{(4)}\,,
\\
\,
\\
\ds\bw^{(4)} \, =\, \bw^{n} \,+\, \frac{\beta\Delta
    t}{\eps}\,\bF^{(1)}\,+\, \frac{\eta\Delta t}{\eps}\,\bF^{(2)}\,+\,
  \frac{\gamma\Delta t}{\eps}\,\bF^{(3)}\,+\, \frac{\alpha\Delta
    t}{\eps}\,\bF^{(4)},
\end{array}\right.
\end{equation}
with 
$$
\left\{
\begin{array}{l}
\ds \bF^{(4)} \,:=\, \bE\left(t^{n+1/2},\hat{\bx}^{(3)}\right) \,-\, \chi\left(\hat{e}^{(3)},\hat{\bw}^{(3)}\right)\,\nabla_\bx (\ln\left(b\right))\left(t^{n+1/2},\hat{\bx}^{(3)}\right) \,-\,b\left(t^{n+1/2},\hat{\bx}^{(3)}\right)\,\frac{(\bw^{(4)})^\perp}{\eps},
\\
\,
\\
\ds S^{(4)} \,:=\,
\bw^{(4)}\cdot\bE\left({t}^{n+1/2},\hat{\bx}^{(3)}\right)\,,
\end{array}\right.
$$
where $t^{n+1/2}=t^n+\Delta t/2$. The numerical solution at the new
time step is finally given by
\begin{equation}
\label{scheme:4-5}
\left\{
\begin{array}{l}
\bx^{n+1} \,=\, \ds \bx^{n}  \,+\, \frac{\Delta t}{6\eps} \left( \bv^{(2)} \,+\,
  \bv^{(3)}  \,+\, 4\, \bv^{(4)} \right),
\\
\,
\\
e^{n+1} \,=\, \ds e^{n}  \,+\, \frac{\Delta t}{6\eps} \left( S^{(2)} \,+\,
  S^{(3)}  \,+\, 4\, S^{(4)} \right),
\\
\,
\\
\bw^{n+1} \,=\, \ds \bw^{n}  \,+\, \frac{\Delta t}{6\eps} \left( \bF^{(2)} \,+\,
  \bF^{(3)}  \,+\, 4 \,\bF^{(4)} \right).
\end{array}\right.
\end{equation}

We emphasize that in all our schemes at each stage the implicit computation only requires the resolution of a two-by-two linear system. Therefore the computational effort is of the same order as  fully explicit schemes.

\begin{proposition}[Consistency in the
  limit $\eps\rightarrow 0$ for a fixed $\Delta t$]
\label{prop:5}
Let us consider a time step $\Delta t>0$, a  final time $T>0$  and
set $N_T=\lfloor T/\Delta t\rfloor$. Assume that $(\bx^n_\eps,\bw^n_\eps,e^{n}_\eps)_{0\leq
  n\leq N_T}$ is a sequence obtained by \eqref{scheme:4-1}-\eqref{scheme:4-5} and such that
\begin{itemize}
\item for all $1\leq n\leq N_T$, $\left(\bx^n_\eps,\eps \bw^n_\eps, e^n_\eps\right)_{\eps>0}$ is uniformly bounded with respect to $\eps>0$;
\item $\left(\bx^0_\eps,\bw_\eps^0,e^0_\eps\right)_{\eps>0}$ converges in the limit $\eps\rightarrow 0$ to some $(\by^0,\bw^0,g^0)$. 
\end{itemize}
Then, for any $1\leq n\leq N_T$, $(\bx^n_\eps,e^n_\eps)_{\eps>0}$ converges to some $(\by^n,g^n)$ as $\eps\rightarrow 0$ and the limiting sequence $(\by^n,g^n)_{0\leq n\leq N_T}$ solves 
\begin{equation}
\left\{
\begin{array}{l}
\ds\by^{n+1} = \by^{n} \,+\, \frac{\Delta t}{6} \,\left( \bU^{n} \,+\, \bU^{(1)} \,+\, 4\,\bU^{(2)}\right)\\\ds
g^{n+1} = g^{n} \,+\, \frac{\Delta t}{6} \,\left( T^{n} \,+\, T^{(1)} \,+\, 4\,T^{(2)}\right)
\label{sch:y40}
\end{array}\right.
\end{equation}
where 
$$
\left\{
\begin{array}{l}
\ds \bU^{n} = -\frac{1}{b(t^n,\by^n)} \Big( \bE(t^n,\by^n) - g^n \,\nabla_\by(\ln(b))(t^n,\by^n)\Big)^\perp,
\\
\,
\\
\ds T^{n} = g^n \,\frac{\nabla_\by^\perp b}{b^2}(t^n,\by^n) \cdot \bE(t^n,\by^n),
\end{array}\right.
$$
whereas $\bU^{(1)}$ and $T^{(1)}$ are given by
$$
\left\{
\begin{array}{l}
\ds \bU^{(1)} = -\frac{1}{b(t^{n+1},\by^{(1)})} \Big( \bE(t^{n+1},\by^{(1)}) - g^{(1)} \,\nabla_\by(\ln(b))(t^{n+1},\by^{(1)})\Big)^\perp,
\\
\,
\\
\ds T^{(1)} = g^{(1)} \,\frac{\nabla_\by^\perp b}{b^2}(t^{n+1},\by^{(1)}) \cdot \bE(t^{n+1},\by^{(1)})
\end{array}\right.
$$
and $\bU^{(2)}$ and $T^{(2)}$
$$
\left\{
\begin{array}{l}
\ds \bU^{(2)} = -\frac{1}{b(t^{n+1/2},\by^{(2)})} \Big( \bE(t^{n+1/2},\by^{(2)}) - g^{(2)} \,\nabla_\by\ln(b (t^{n+1/2},\by^{(2)}))\Big)^\perp,
\\
\,
\\
\ds T^{(2)} = g^{(2)} \,\frac{\nabla_\by^\perp b}{b^2}(t^{n+1/2},\by^{(2)}) \cdot \bE(t^{n+1/2},\by^{(2)})
\end{array}\right.
$$
with $t^{n+1/2}=t^n+\Delta t/2$,
$$
\by^{(1)}\, =\, \by^n \,+\, \Delta t\, \bU^n,\qquad g^{(1)} \,=\, g^n \,+\, \Delta t\, T^n\,,
$$
and
$$
\by^{(2)}\, =\, \by^n \,+\, \frac{\Delta t}{4}\,
\left(\bU^n+\bU^{(1)}\right),\qquad g^{(2)} \,=\, g^n \,+\,
\frac{\Delta t}{4}\, \left( T^n + T^{(1)}\right),
$$
which provides a consistent third-order approximation of the gyro-kinetic system \eqref{traj:limit}.
\end{proposition}

We omit the proof of Proposition~\ref{prop:5} as almost identical to the one of Proposition~\ref{prop:3}.

\begin{remark}
For both our second and third-order schemes, the limit obtained when $\eps\to0$  fails to satisfy the first step of the limiting schemes by a first-order error. This is obviously a more serious problem that in the first-order case. However, though we do not pursue this line of investigation here, one should be able to fix this issue by replacing the very first step (from $t^0$ to $t^1$) of the original ($\eps$-dependent) computation by a step where third equations of intermediate values are also implicited.
\end{remark}

\section{Numerical simulations}
\label{sec:5}
\setcounter{equation}{0}

In this section, we provide examples of numerical computations to validate and compare
the different time discretization schemes introduced in the previous section. 

We first consider the motion of a single particle under the effect of a given
electromagnetic field. It allows us to  illustrate the ability of the semi-implicit schemes
to capture in the limit $\varepsilon\rightarrow 0$ drift velocities due to variations of magnetic and electric fields, even with large time steps $\Delta t $ .

Then we consider the Vlasov-Poisson system with an external non uniform magnetic field.  As already mentioned we implement a classical particle-in-cell method but with specific different time discretization techniques to compute the particle trajectories. In particular, in this case, the collection of charged particles moves collectively and gives rise to a self-consistent electric field, obtained by solving numerically the Poisson equation in a space grid. 

\subsection{One single particle motion without electric field}
Before simulating at the statistical level, we investigate on the motion of individual particles in a given magnetic field the accuracy and stability properties with respect to $\eps>0$ of the semi-implicit algorithms presented in Section~\ref{sec:3}.  

Numerical experiments of the present subsection are run with a zero electric field $\bE=0$, and a time-independent external magnetic field corresponding to 
$$
b\,:\quad \RR^2\to\RR\,,\qquad \bx=(x,y)\,\mapsto\,1\,+\,\alpha\,x^2
$$
with $\alpha=0.5$. Moreover we choose for all simulations the initial data as
$\bx^0=(5,4)$, $\bv^0=(5,6)$, whereas the final time is $T=2$. In this case, 
the asymptotic drift velocity predicted by the limiting model~\eqref{traj:limit} is explicitly given by 
$$
\bU(\bx,e) \,=\, \frac{2\,\alpha\,e}{(1+\alpha \,x^2)^2}
\left(\begin{array}{l} 0\\x\end{array}\right).
$$

First, for comparison, we compute a reference solution $(\bX^\eps,\bw^\eps,e^\eps)_{\eps>0}$
to the initial problem \eqref{traj:ter} thanks to an explicit
fourth-order Runge-Kutta scheme used with a very small time step of
the order of $\eps^2$ and a reference solution $(\bY,g)$ to the (non
stiff) asymptotic model \eqref{traj:limit} obtained when $\eps
\rightarrow 0$. Recall that the derivation of \eqref{traj:limit} also
shows weak convergence of $\eps^{-1}\bw^\eps$ to $\bU^0=\bU(\bY,g)$ in
the limit $\eps\to0$. Then we compute an approximate solution
$(\bX^\eps_{\Delta t}, \bw^\eps_{\Delta t}, e^\eps_{\Delta t})$  using either \eqref{scheme:3-1}--\eqref{scheme:3-3} or \eqref{scheme:4-1}--\eqref{scheme:4-5}, and compare them to the reference solutions. 

Our goal is  to evaluate the accuracy of the numerical solution
$(\bX^\eps_{\Delta t}, \bw^\eps_{\Delta t}, e^\eps_{\Delta t})$ for various regimes when both $\eps$ and $\Delta t$ vary. Computed errors are measured in discrete $L^1$ norms. For instance, we set\footnote{When necessary and not too confusing, as here, we use the same piece of notation for continuous time-dependent functions and their discrete counterpart, obtained by restriction to discrete times.}
$$
\left\{
\begin{array}{ll}
\ds\|\bX^\eps_{\Delta t} - \bX^\eps\|\,:=\, \frac{\Delta t
}{T}\sum_{n=0}^{N_T}  \|\bX^{\eps,n}_{\Delta t} - \bX^\eps(t^n) \|\,,
\\[0.5em]
\ds\eps^{-1}\,\ds\|\bw^\eps_{\Delta t} - \bw^\eps\|\,:=\, \frac{\Delta t
}{\eps\, T}\sum_{n=0}^{N_T}  \|\bw^{\eps,n}_{\Delta t} - \bw^\eps(t^n) \|\,.
\end{array}\right.
$$

In Figure~\ref{fig0:1}, we present the numerical error on space and velocity\footnote{In graphics we use $\bv$ as a short-hand for $\eps^{-1}\bw$. This should not be confused with the actual reconstruction of the velocity variable used at the statistical level.} variables between the reference solution for the initial problem \eqref{traj:00} and the one  obtained with the third-order scheme \eqref{scheme:4-1}-\eqref{scheme:4-5}. As expected for a fixed time step $\Delta t$ (taken here between $0.0025$ and $0.01$), the scheme is quite stable even in the limit $\eps\rightarrow 0$ and the error on the space variable, measured by $\|\bX^\eps_{\Delta t} - \bX^\eps\|$, is uniformly small and reach a maximum on intermediate regimes, $\eps \in (0.05,\;0.5)$. In contrast, for a fixed time step, the error on the velocity variable, measured by $\eps^{-1}\,\|\bw^\eps_{\Delta t} - \bw^\eps\|$, is small  for large values of $\eps$, but gets very large when $\eps\rightarrow 0$ since the scheme cannot follow high-frequency time oscillations of order $\eps^{-2}$ when $\eps\ll\sqrt{\Delta t}$.
 
\begin{figure}[ht!]
\begin{center}
 \begin{tabular}{cc}
\includegraphics[width=8.cm]{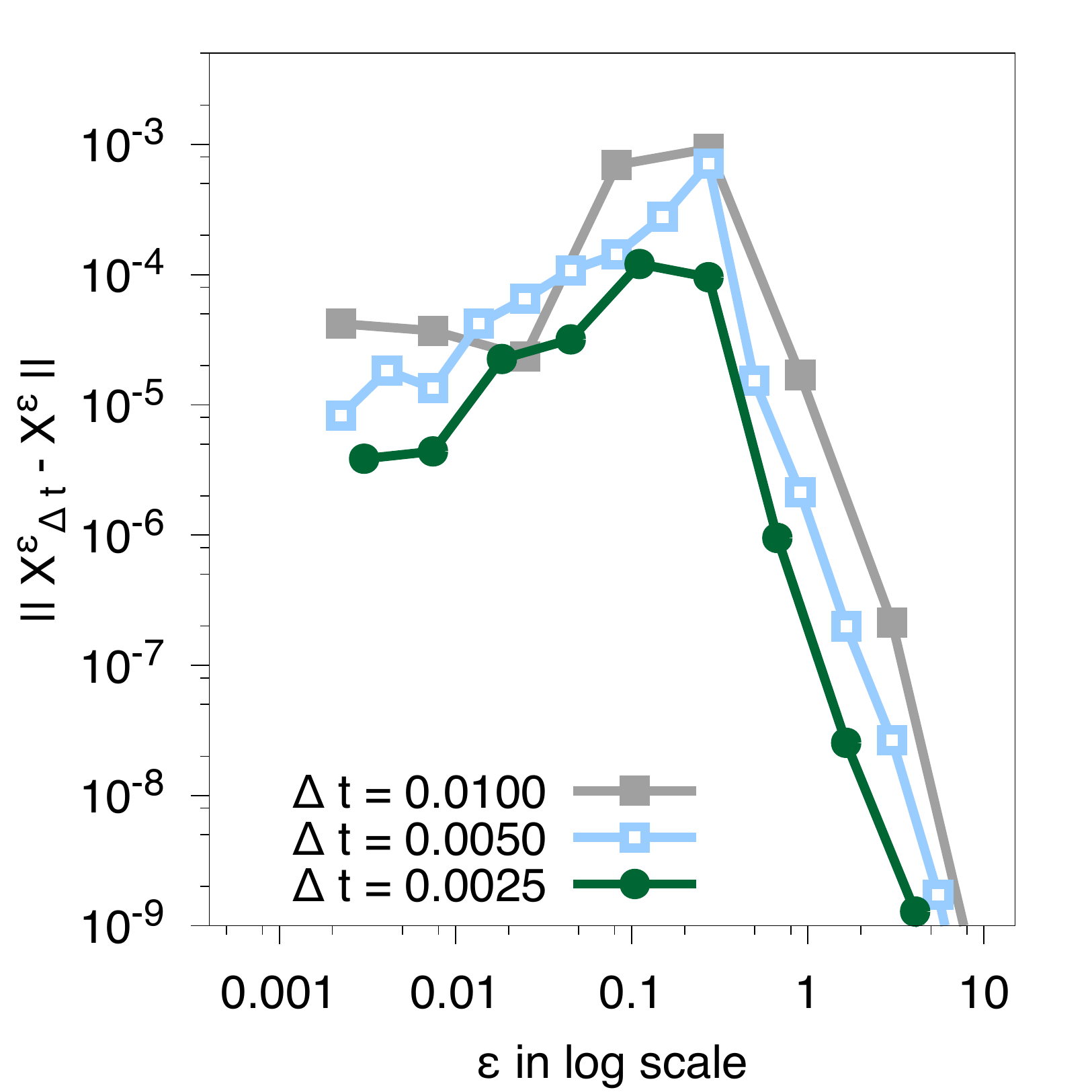} &    
\includegraphics[width=8.cm]{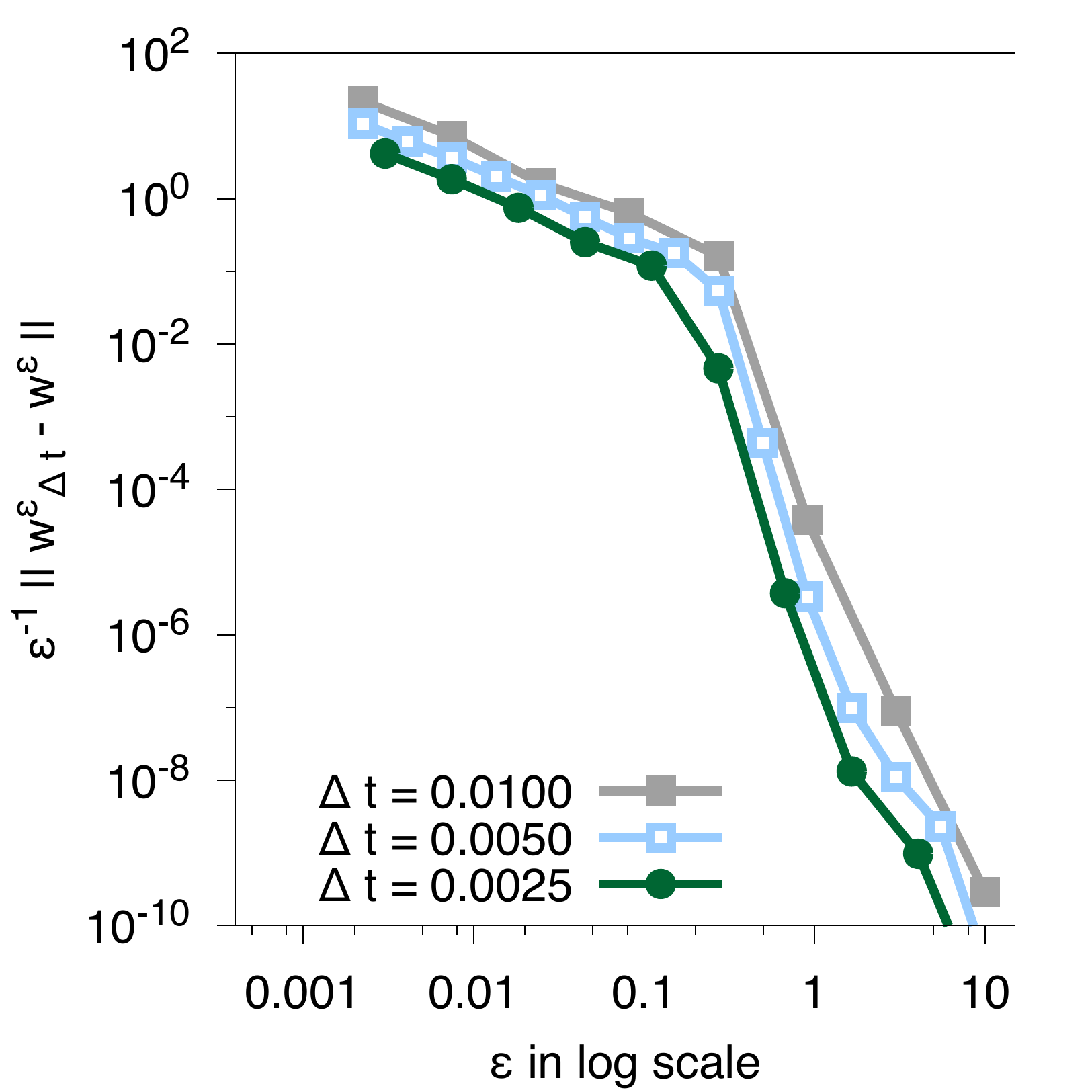} 
\\
(a)  & (b)  
\end{tabular}
\caption{\label{fig0:1}
{\bf One single particle motion without electric field.} Numerical errors (a)
$\|\bX^\eps_{\Delta t}-\bX^\eps\|$,  (b)
$\eps^{-1}\,\|\bw^\eps_{\Delta t}-\bw^\eps\|$ obtained for different time steps $\Delta t$ with third-order scheme \eqref{scheme:4-1}-\eqref{scheme:4-5}, plotted as functions of $\eps$.}
 \end{center}
\end{figure}

In Figure \ref{fig0:2}, for the same set of simulations we plot errors with respect to the limiting system~\eqref{traj:limit}. This shows, see Figure \ref{fig0:2} (b), that for a fixed $\Delta t$ we do capture however, in the limit $\eps\to0$, correct drift velocities. More generally, as expected, errors on both space and velocity variables with respect to limiting asymptotic values gets smaller and smaller as $\eps\to0$. 

\begin{figure}[ht!]
\begin{center}
 \begin{tabular}{cc}
\includegraphics[width=8.cm]{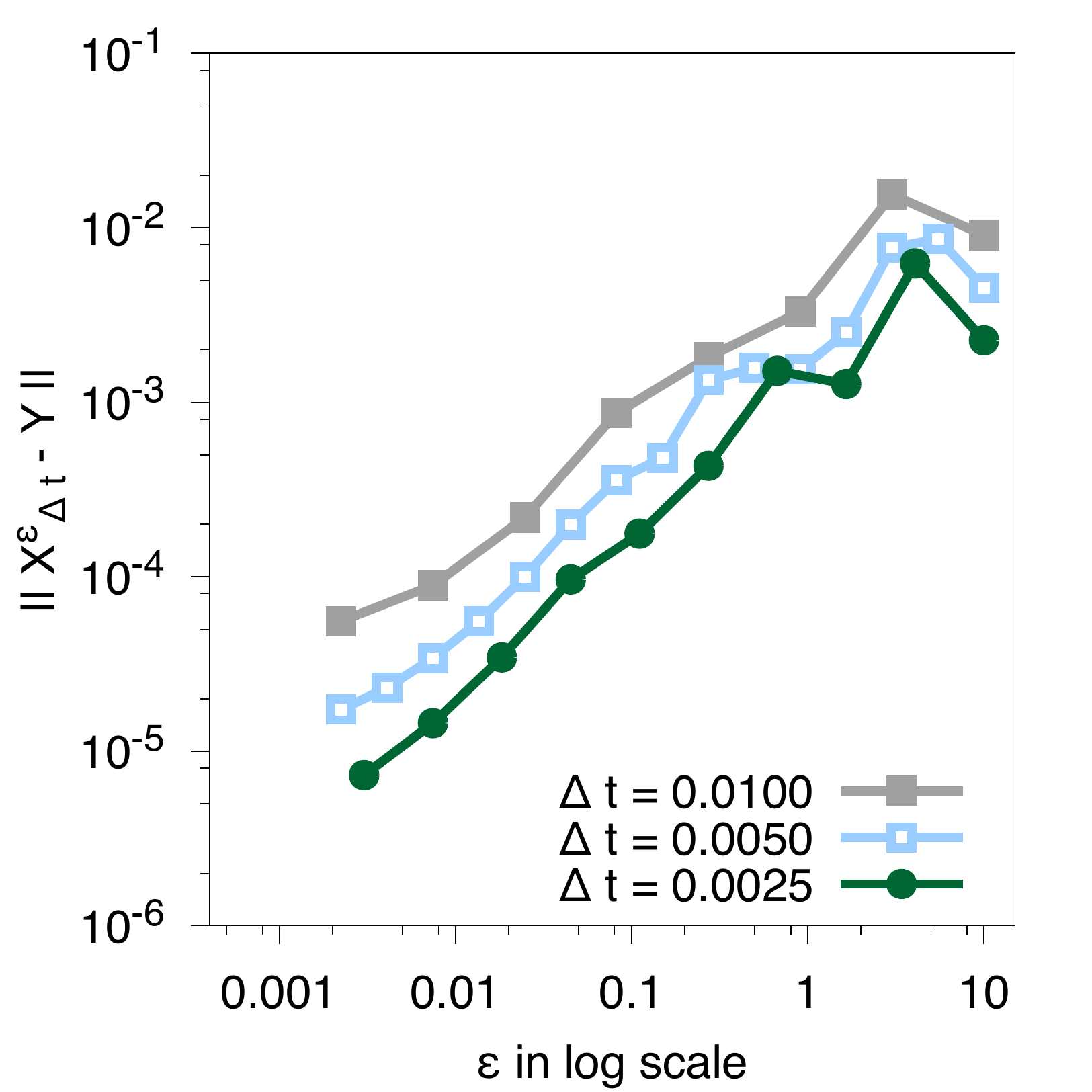} &    
\includegraphics[width=8.cm]{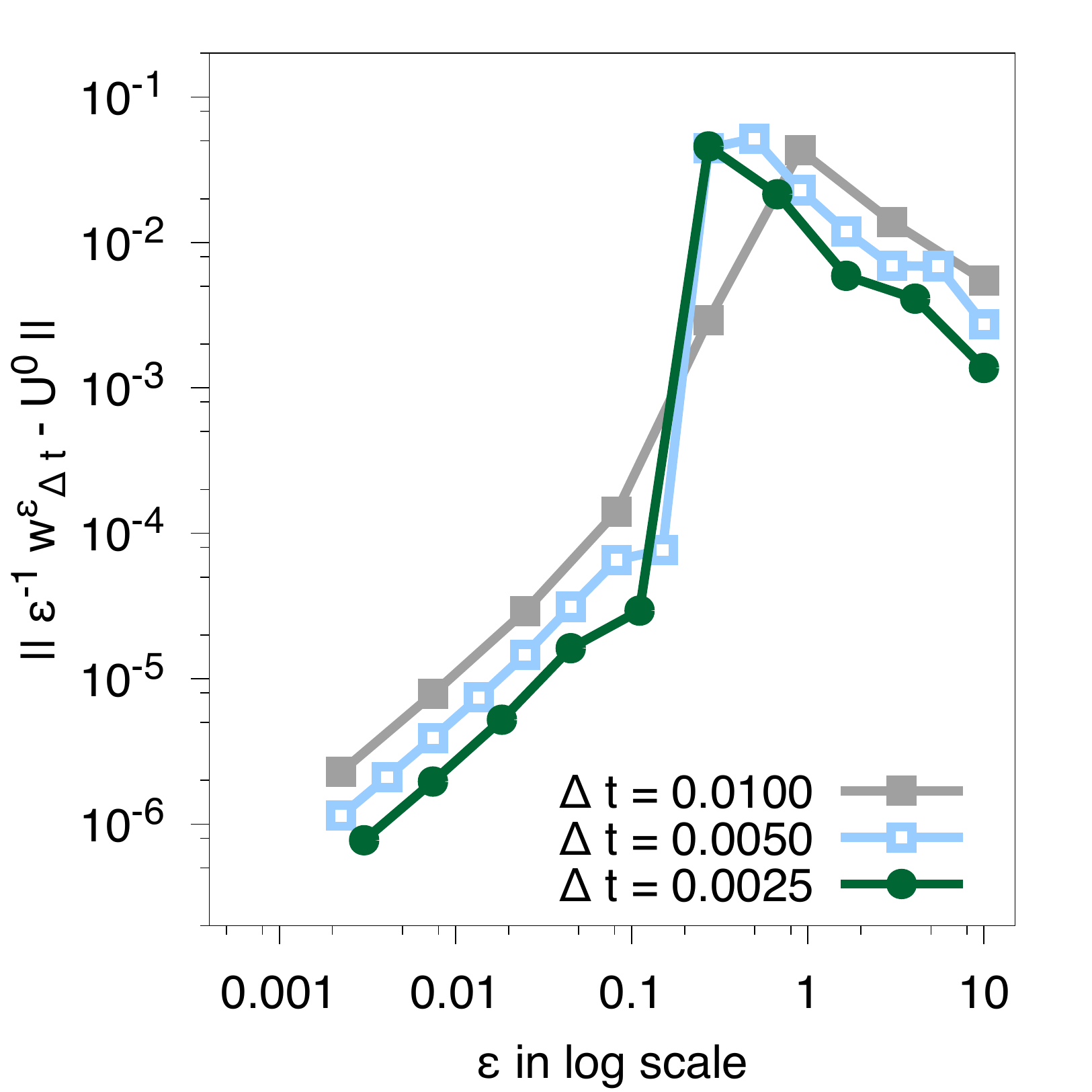} 
\\
(a)  & (b)  
\end{tabular}
\caption{\label{fig0:2}
{\bf One single particle motion without electric field.} Numerical errors (a)
$\|\bX^\eps_{\Delta t}-\bY\|$,  (b)
$\|\eps^{-1}\,\bw^\eps_{\Delta t}-\bU^0\|$ obtained for different time steps $\Delta t$ with third-order scheme \eqref{scheme:4-1}-\eqref{scheme:4-5}, plotted as functions of $\eps$.}
 \end{center}
\end{figure}

Finally, in Figure~\ref{fig0:3} we show similar quantities for the second-order scheme \eqref{scheme:3-1}-\eqref{scheme:3-3} holding the time step fixed to $\Delta t = 0.01$. For comparison, we also plot results of the third-order scheme. Errors of the third-order scheme are smaller but behaviors are qualitatively the same : uniform precision on the slow variable, convergence to asymptotic descriptions as $\eps\to0$ and consistency of the velocity variable with the velocity of the true solution when $\eps\gg\sqrt{\Delta t}$ and with the asymptotic drift when $\eps\ll\sqrt{\Delta t}$.

\begin{center}
\begin{figure}[ht!]
 \begin{tabular}{cc}
\includegraphics[width=8.cm]{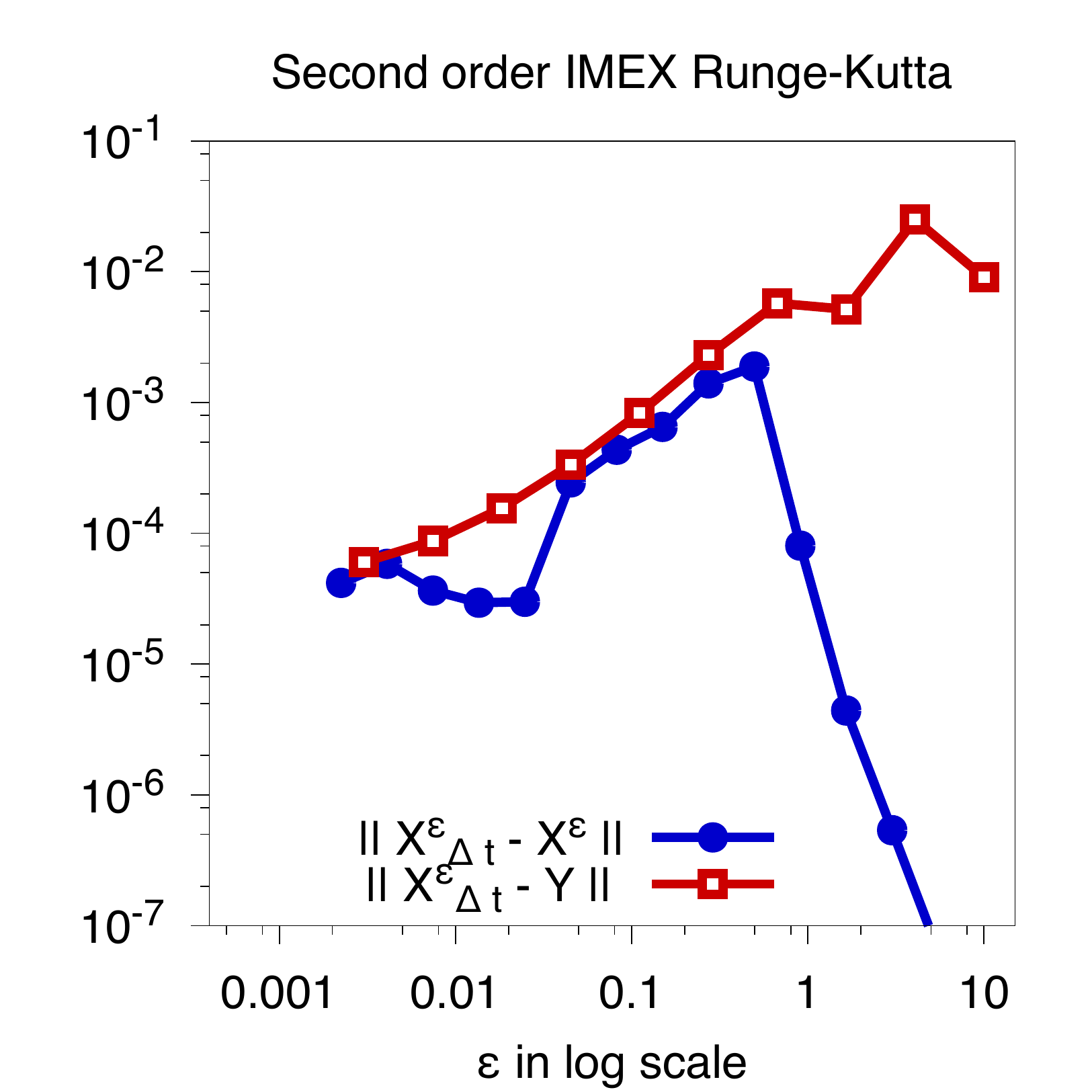} &    
\includegraphics[width=8.cm]{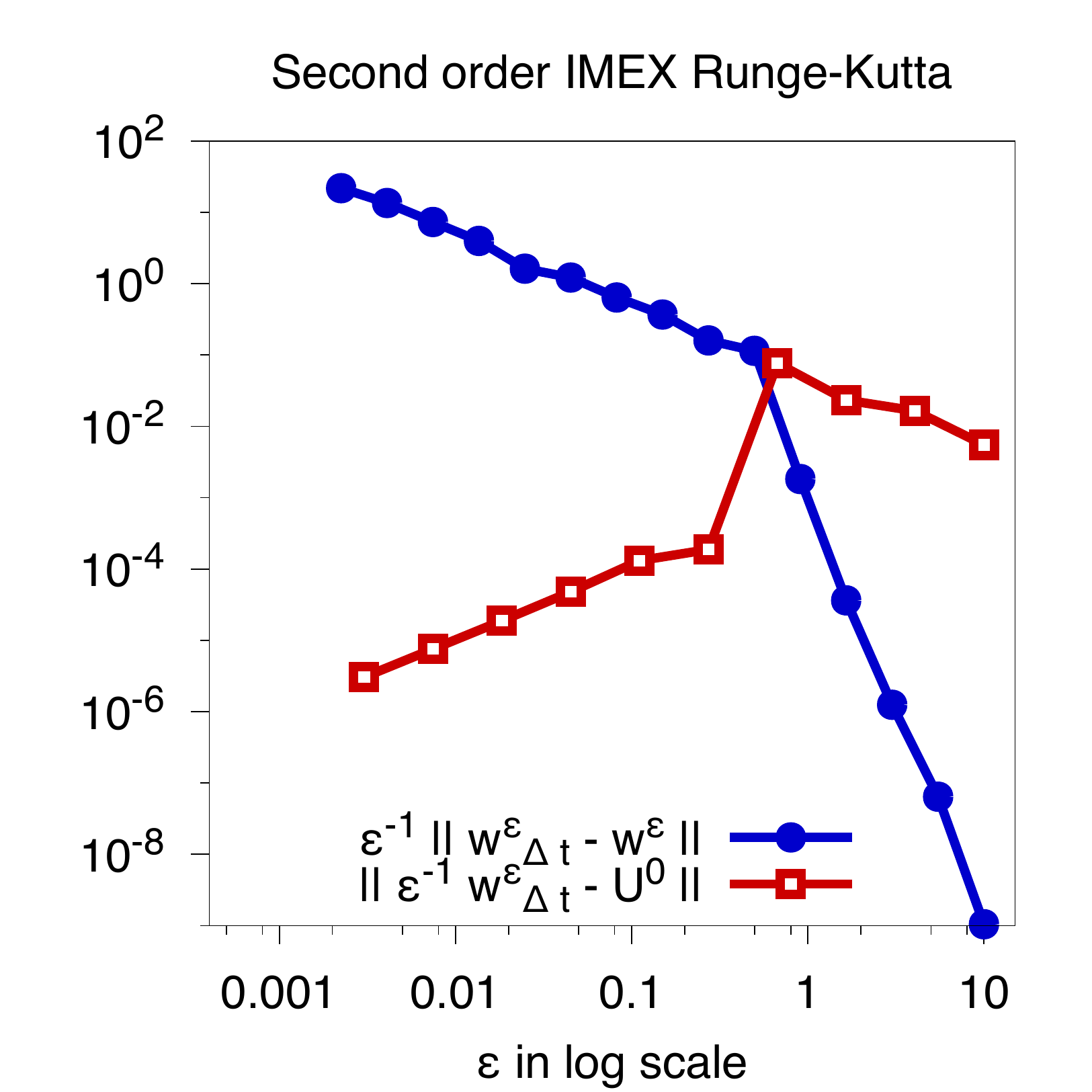}  
 \\
\includegraphics[width=8.cm]{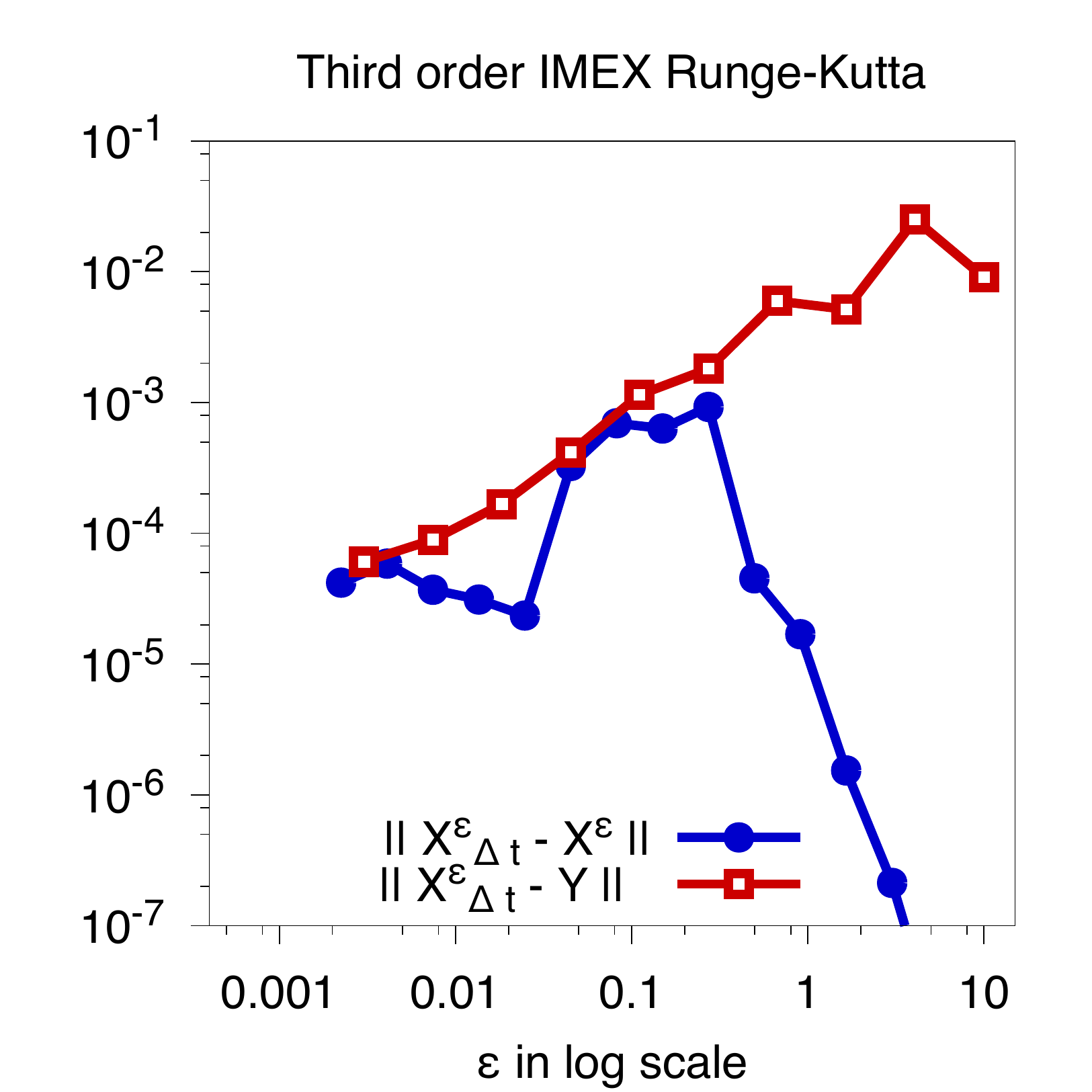} &    
\includegraphics[width=8.cm]{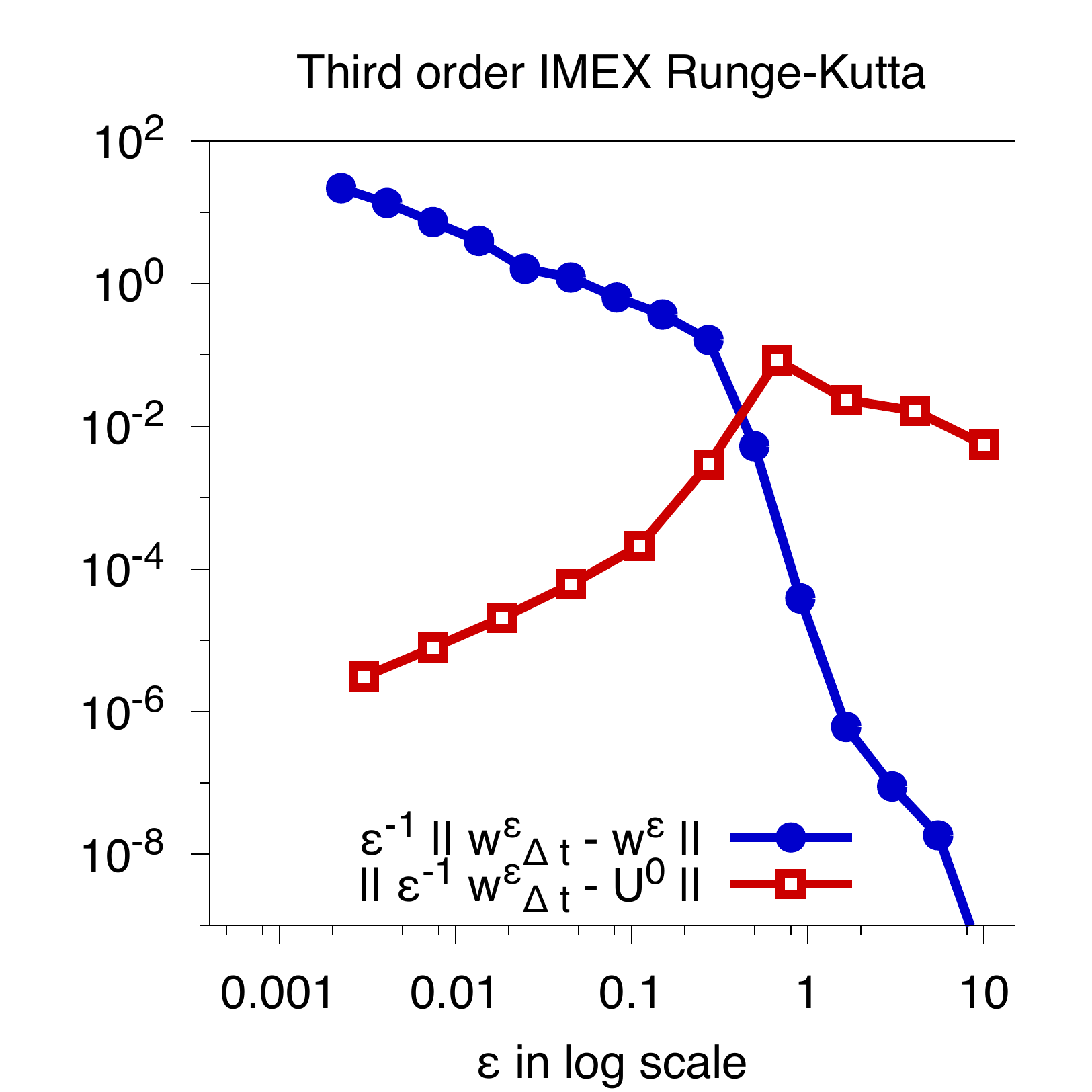}  
\\ 
(a)  & (b)  
\end{tabular}
\caption{\label{fig0:3}
{\bf One single particle motion without electric field.} Numerical errors on (a)
$(\bX^\eps_{\Delta t})_{\eps>0}$  and (b) $(\eps^{-1}\,\bw^\eps_{\Delta t})_{\eps>0}$ obtained with second-order scheme \eqref{scheme:3-1}-\eqref{scheme:3-3} and third-order scheme \eqref{scheme:4-1}-\eqref{scheme:4-5}, plotted as functions of $\eps$.}
\end{figure}
\end{center}

\subsection{One single particle motion with an electromagnetic field}
We now add to the otherwise unmodified previous setting a non trivial electric field 
$$
\bE\,:\quad \RR^2\to\RR\,,\qquad \bx=(x,y)\,\mapsto\,(0,-y)\,.
$$
In this case, at the limit $\eps\to0$ the behavior of the microscopic kinetic energy remains non trivial (in contrast with what happens in the previous subsection) so that we also measure errors on its evaluations. Moreover the electric field also contributes to asymptotic drift velocities, that are now given by
$$
\bU^0(\bx,e) \,=\, \frac{1}{(1+\alpha\,x^2)} \left(\begin{array}{l} -y \\\,\\
                                           \ds\frac{2\,\alpha\,e\,x}{1+\alpha\,x^2}\end{array}\right),
$$
with $\alpha=1/2$. This simple geometric configuration allows us to clearly distinguish the contribution to the drift velocity due to the interaction of the electric field with the magnetic field from purely magnetic effects due to the gradient of the magnetic field.

Quantitative error measurements are provided in Figure~\ref{fig1:1}. Qualitative features are completely analogous to those of the previous subsection, the microscopic kinetic energy being approximated with essentially the same accuracy as the other slow variable $\bx$.

\begin{center}
\begin{figure}[ht!]
 \begin{tabular}{ccc}
\includegraphics[width=5.3cm]{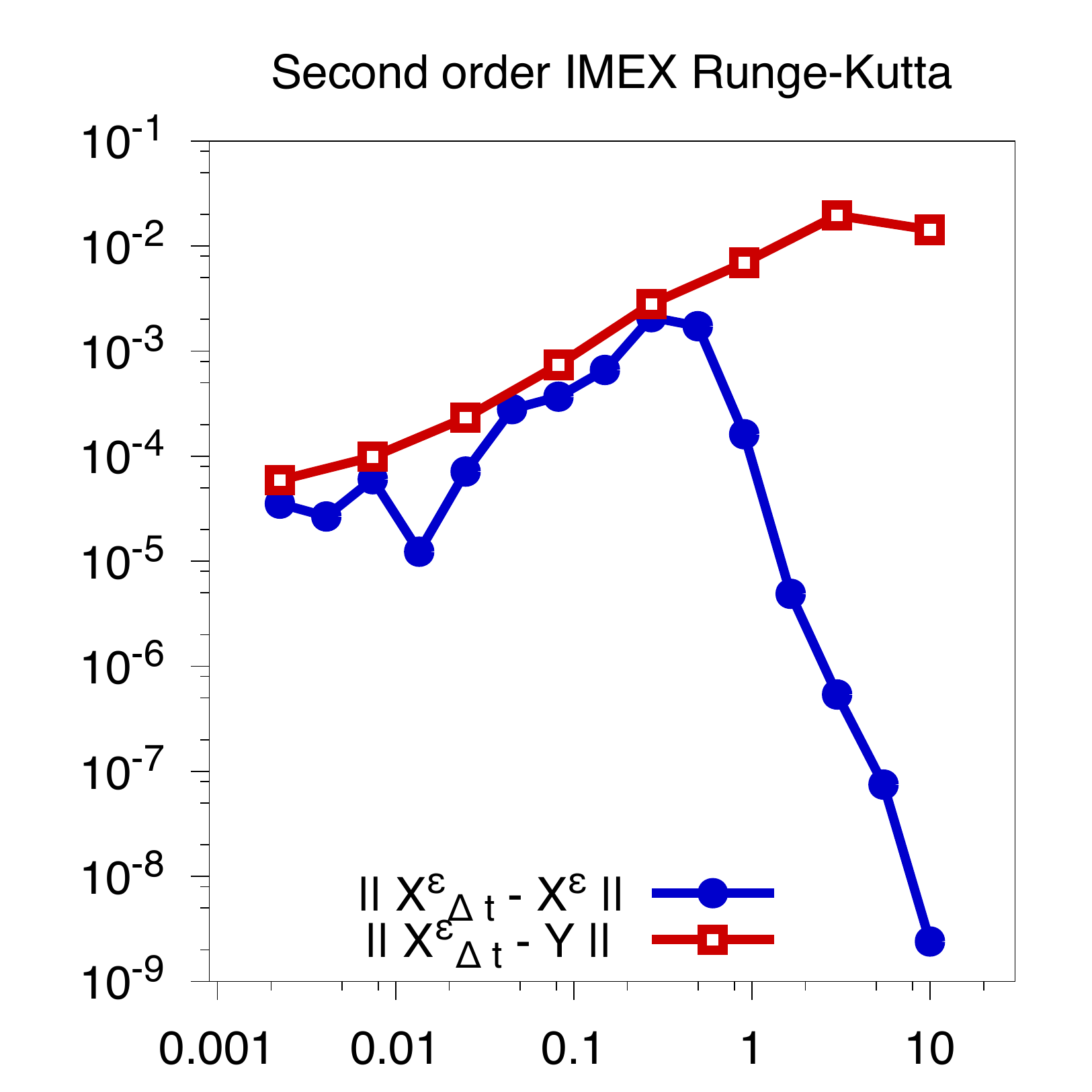} &   
\includegraphics[width=5.3cm]{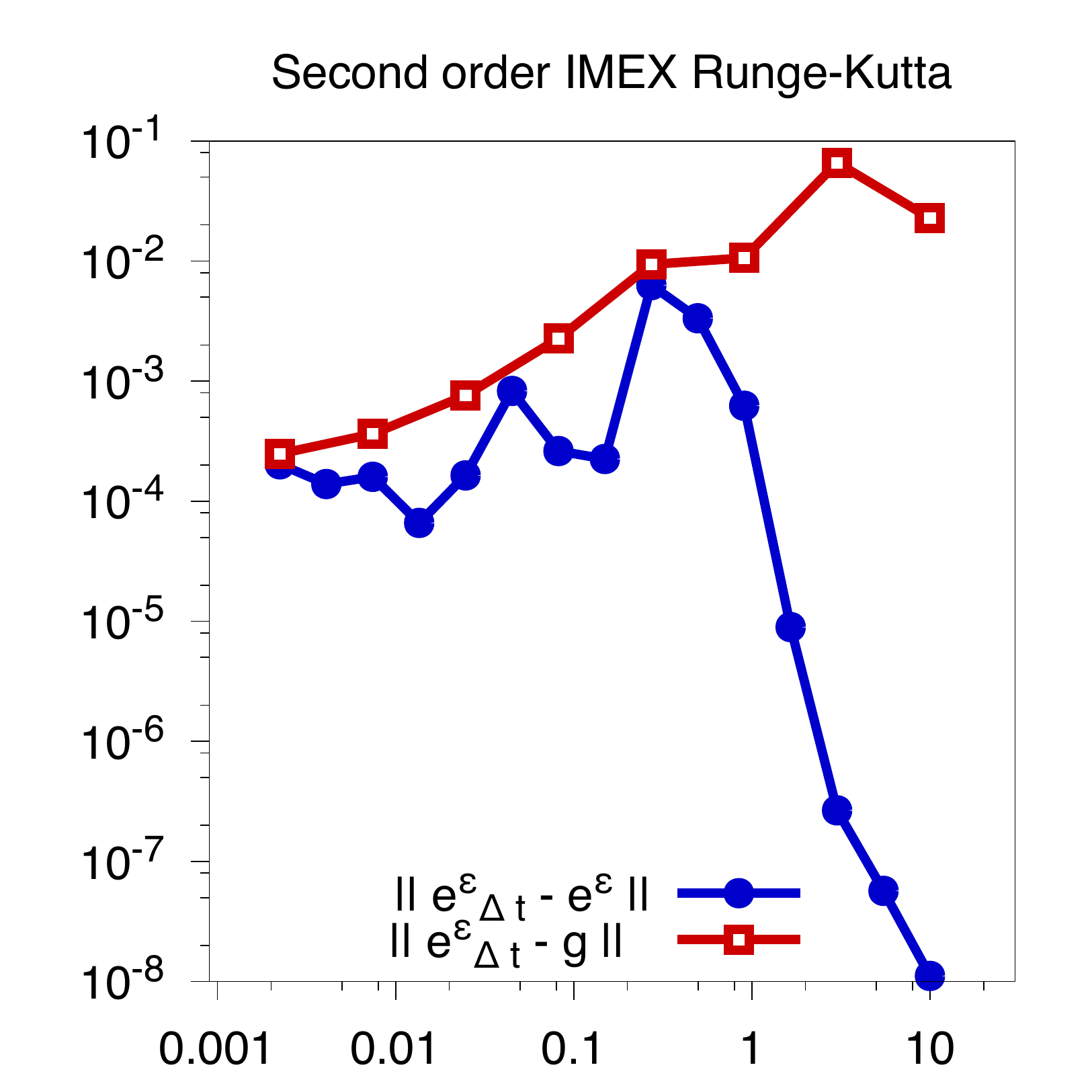} &    
\includegraphics[width=5.3cm]{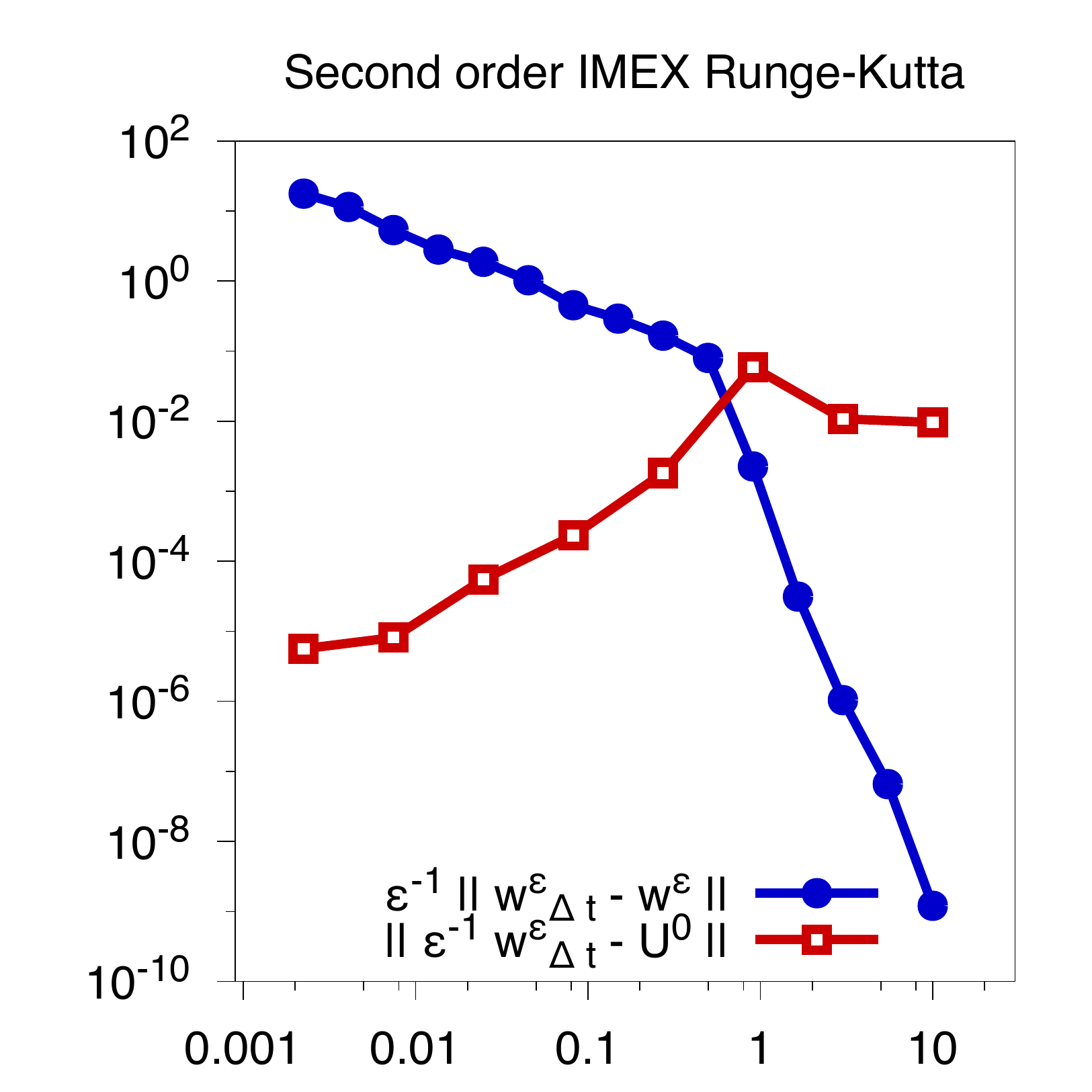}  
 \\
\includegraphics[width=5.3cm]{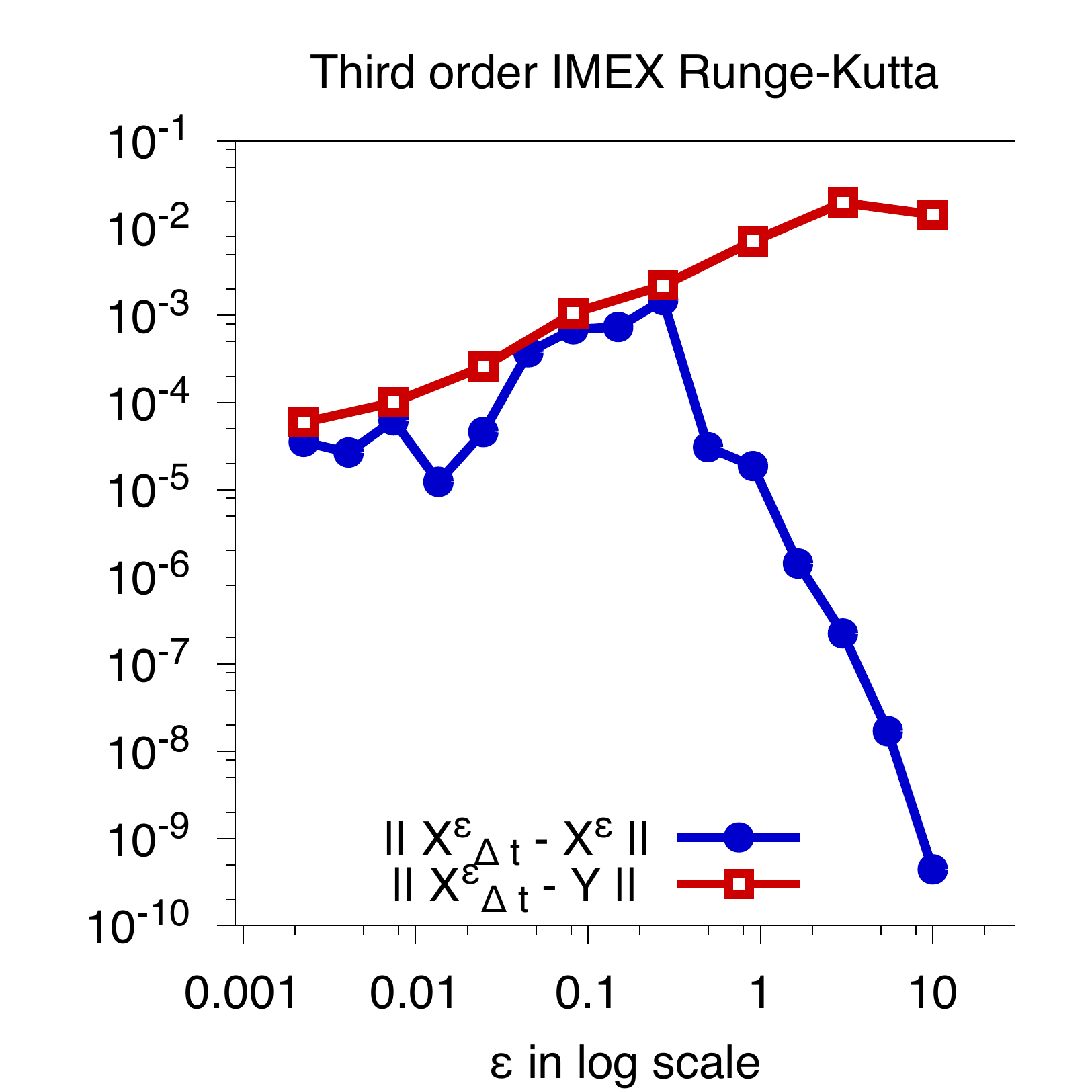} &    
\includegraphics[width=5.3cm]{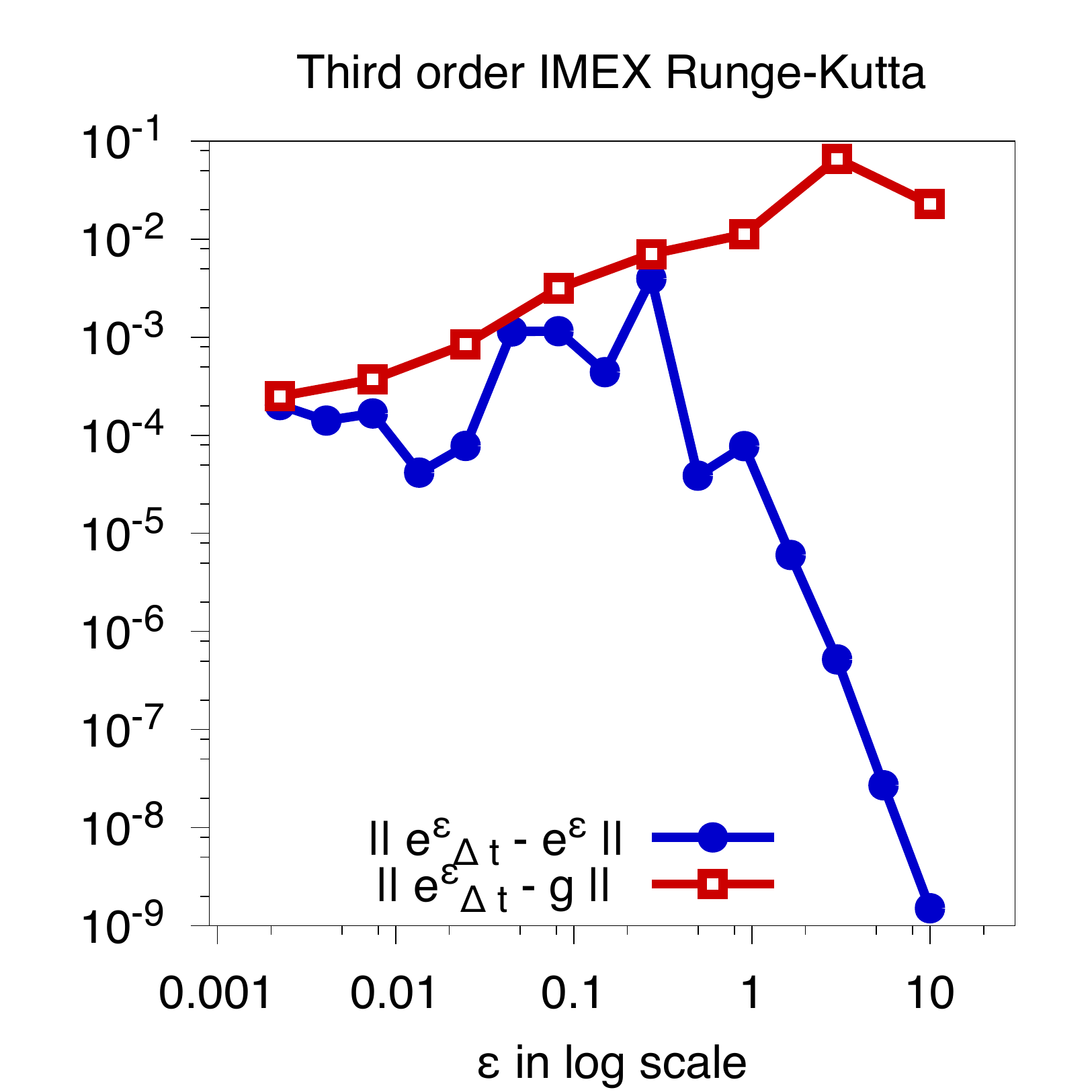} & 
\includegraphics[width=5.3cm]{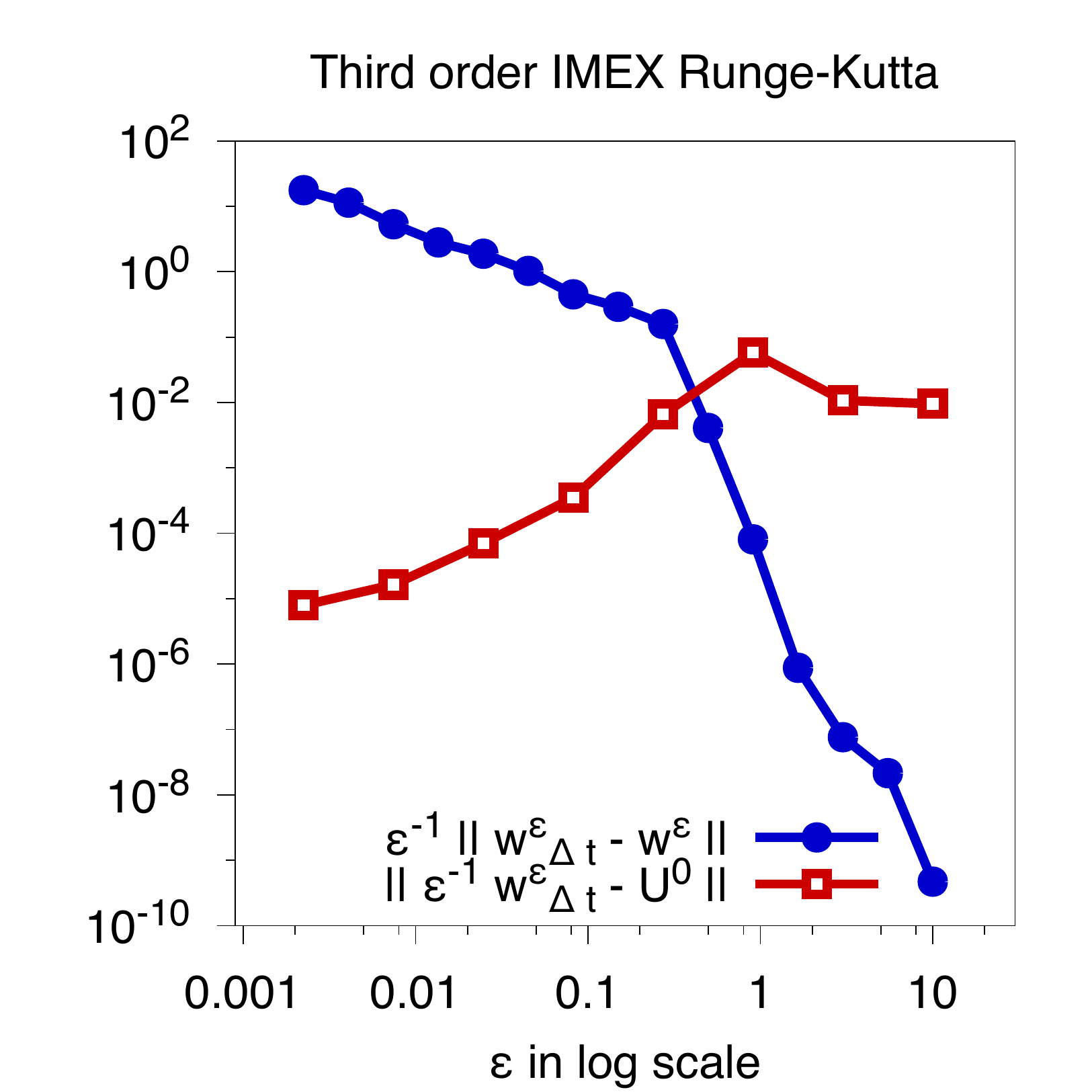}  
\\ 
(a)  & (b)   & (c)  
\end{tabular}
\caption{\label{fig1:1}
{\bf One single particle motion with an  electromagnetic field.} Numerical errors on (a)
$(\bX^\eps_{\Delta t})_{\eps>0}$,  (b) $(e^\eps_{\Delta
  t})_{\eps>0}$ and (c) $(\eps^{-1}\,\bw^\eps_{\Delta t})_{\eps>0}$ obtained with the second-order scheme
\eqref{scheme:3-1}-\eqref{scheme:3-3} and third-order scheme
\eqref{scheme:4-1}-\eqref{scheme:4-5}, plotted as functions of $\eps$.}
\end{figure}
\end{center}

To visualize rather than quantify what happens in the limit $\eps\to0$, we also represent now space trajectories corresponding to different values of $\eps$, but holding the time step fixed to $\Delta t=0.01$. On the reference solution one sees as $\eps\to0$ faster and faster oscillations --- of order $\eps^{-2}$ --- but with smaller and smaller amplitude --- of order $\eps$ --- around guiding center trajectories as predicted by\footnote{Note incidentally that this illustrates a classical abuse in terminology, used throughout the present paper. Variables $(\bX^\eps,e^\eps)$ are not really slow but they are uniformly close to genuinely slow variables.} \eqref{traj:limit}. When $\eps\gg0.1$ the approximate solution follows closely the reference solution, whereas when $\eps\ll0.1$ they are also almost indistinguishable though the approximate solution does not reproduce the infinitely fast oscillations of the true solution. Main discrepancies are observed for intermediate values, here exemplified by the case $\eps=0.2$, when the true solution is already too fast for the approximate solution but not yet close enough to the guiding center dynamics. 

\begin{center}
\begin{figure}[ht!]
 \begin{tabular}{cc}
\includegraphics[width=8.cm]{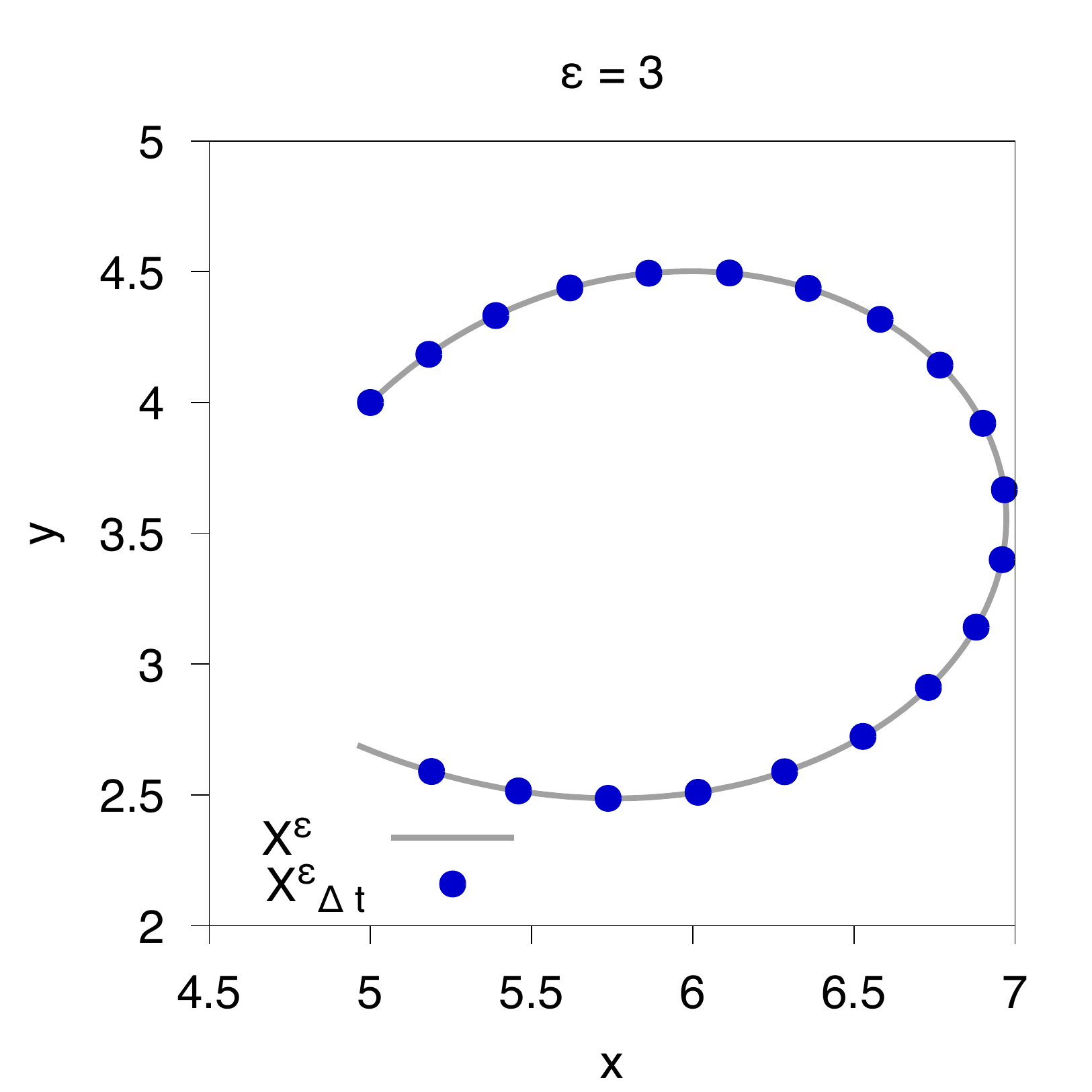} &   
\includegraphics[width=8.cm]{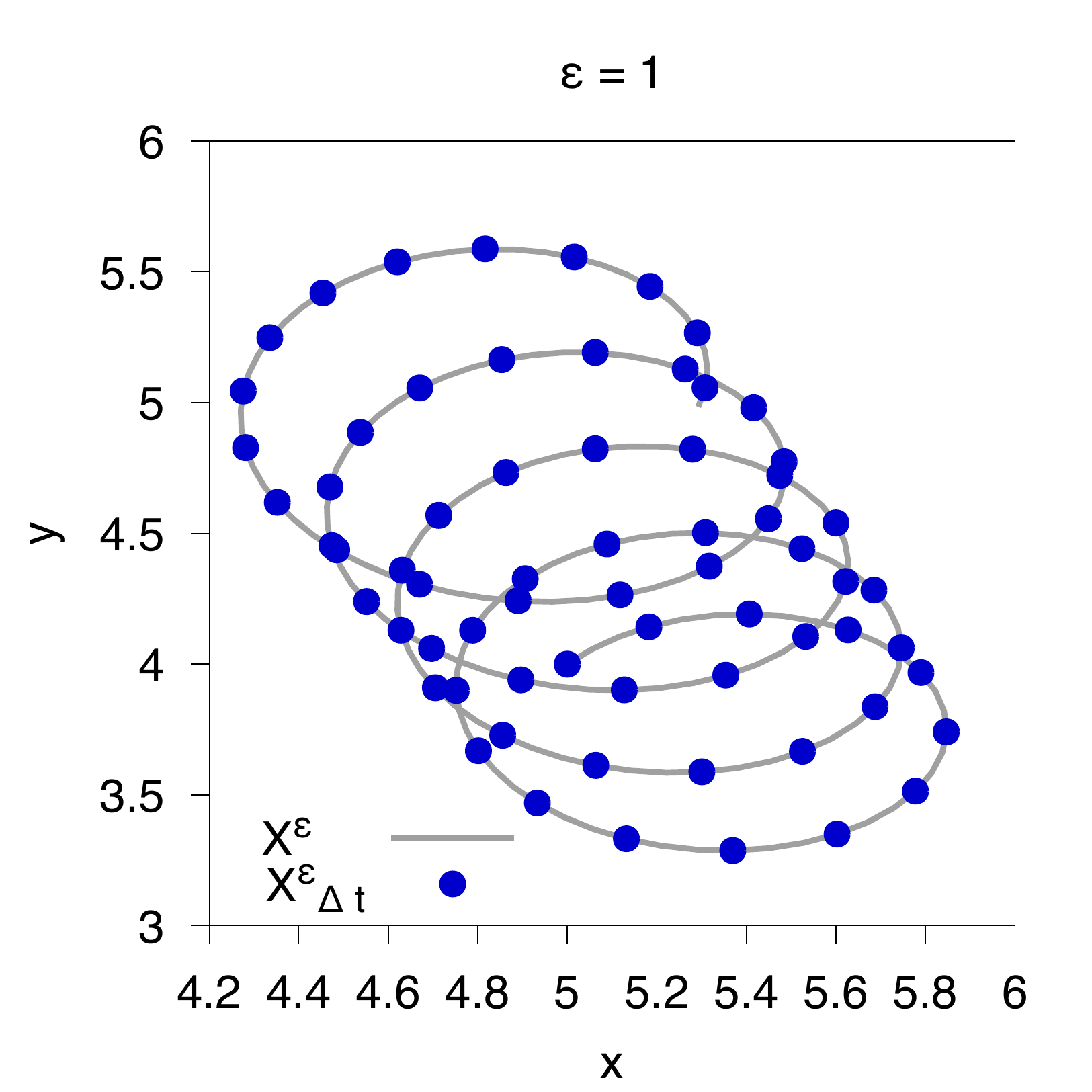} 
\\ 
(a)  & (b)  
\\  
\includegraphics[width=8.cm]{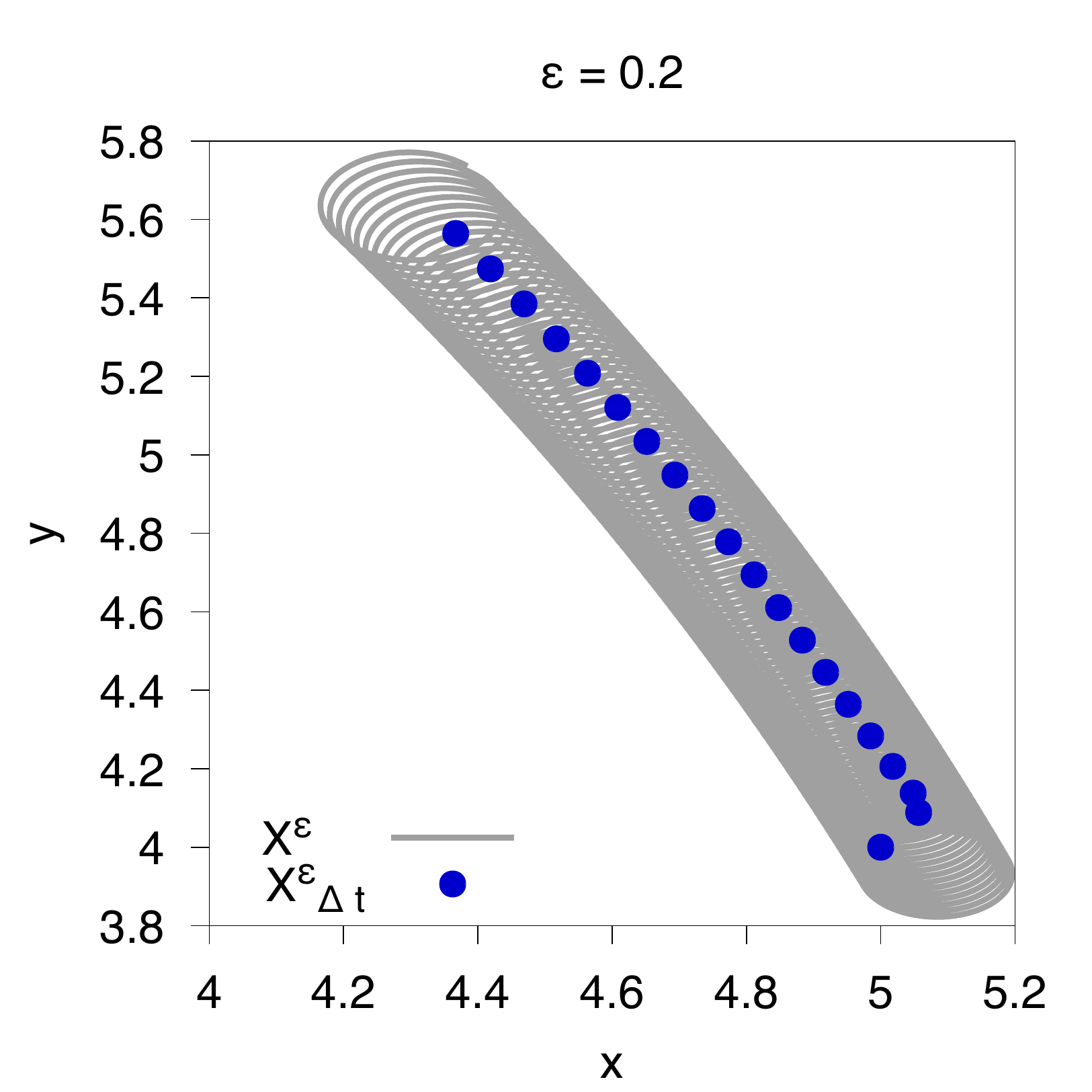} &    
\includegraphics[width=8.cm]{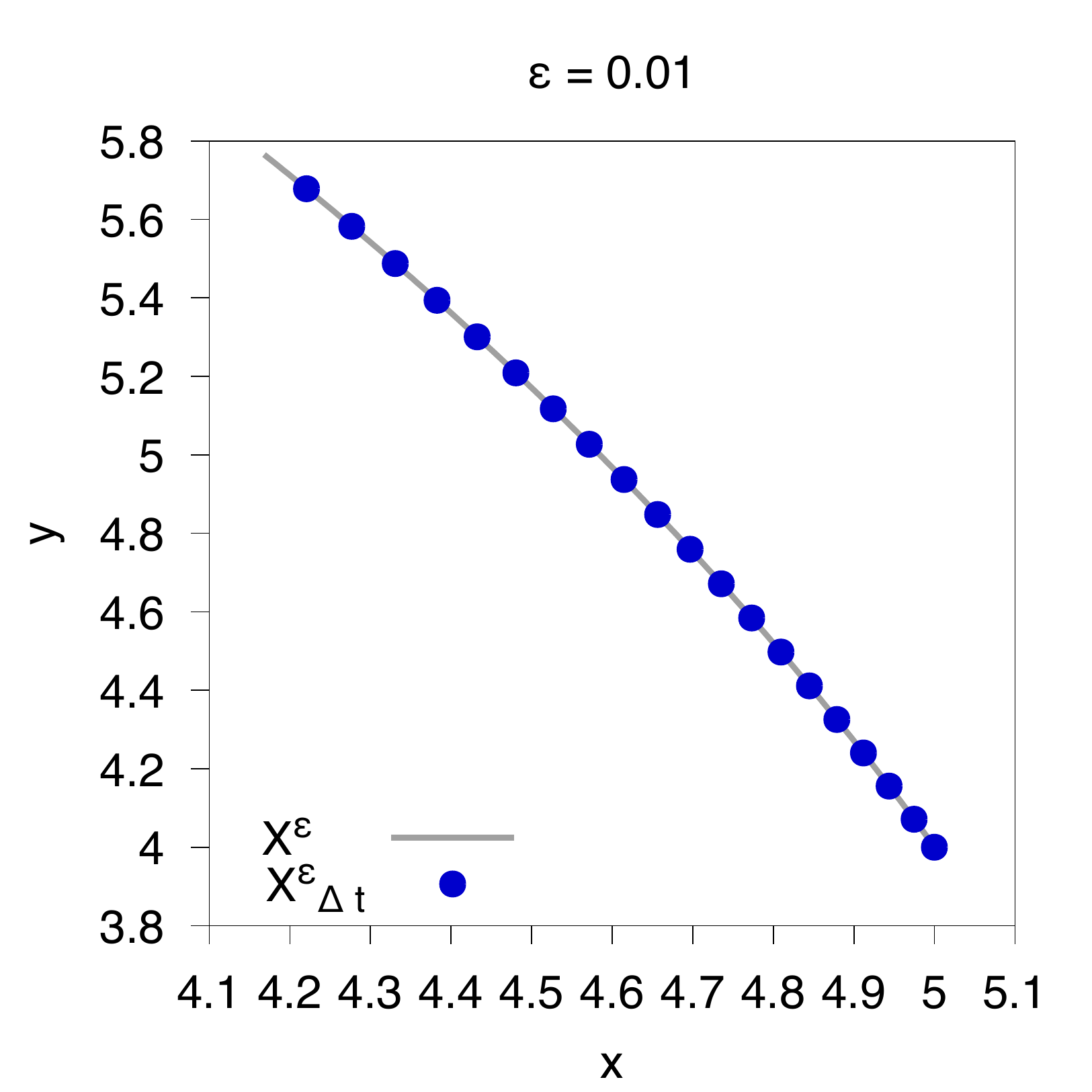}  
\\ 
(c)  & (d) 
\end{tabular}
\caption{\label{fig1:2}
{\bf One single particle motion with an  electromagnetic field.} Space trajectories
$(\bX^\eps_{\Delta t})_{\eps>0}$ for (a) $\eps=3$, (b) $\eps=1$, (c)
$\eps=0.2$ and (d) $\eps=0.01$, computed with third-order scheme \eqref{scheme:4-1}-\eqref{scheme:4-5}.}
\end{figure}
\end{center}

As a  conclusion,  these elementary numerical simulations  confirm the ability of the semi-implicit discretization to capture the evolution of the "slow" variables $(\bX^\eps,e^\eps)$ uniformly with respect to $\eps$ by essentially transitioning automatically to the the guiding center motion encoded by \eqref{traj:limit} when it is not enable anymore to follow the fast oscillations of the true evolution.

\subsection{The Vlasov-Poisson system}
We now consider the Vlasov-Poisson system \eqref{eq:vlasov2d} set in a disk  $\Omega=D(0,6)$ centered at the origin and of radius $6$. Our simulations start with an initial data that is Maxwellian in velocity and whose macroscopic density is the sum of two Gaussians, explicitly 
$$
f_0(\bx,\bv) \,=\, \frac{1}{16\pi^2} \left[
  \exp\left(-\frac{\|\bx-\bx_0\|^2}{2}\right) + \exp\left(-\frac{\|\bx+\bx_0\|^2}{2}\right)\right]\, \exp\left(-\frac{\|\bv\|^2}{4}\right),
$$ 
with $\bx_0=(3/2,-3/2)$. The parameter $\eps$ is set either to $\eps=1$, leading to
a non-stiff problem or to $\eps=0.05$, where the asymptotic regime is relevant. For both regimes we compute numerical solutions to the Vlasov-Poisson system \eqref{eq:vlasov2d}  with the third-order scheme \eqref{scheme:4-1}-\eqref{scheme:4-5} and time step $\Delta t=0.1$. Note that heuristic arguments based on single-particle simulations suggest that when $\eps=0.05$ this enforces a regime where we are close to the asymptotic limit and our schemes cannot capture fast oscillations of the true solution.

We run two sets of numerical simulations with a time-independent inhomogeneous magnetic field
$$
b\,:\quad \RR^2\to\RR\,,\qquad \bx\,\mapsto\,\frac{10}{\sqrt{10^2-\|\bx\|^2}}
$$
that is radial increasing with value one at the origin.

In Figure~\ref{fig2:0} we represent time evolutions of the total energy 
$$
\mathcal{E}(t) \,:=\,
\iint_{\Omega\times\RR^2} f(t,\bx,\bv)\,\frac{\|\bv\|^2}{2}\,\dd\bx\,\dd \bv
\,+\,\frac{1}{2}\int_{\Omega} \|\bE(t,\bx) \|^2 \dd\bx
$$
and of the adiabatic invariant 
$$
\mu(t)=\int_{\Omega}\int_{\RR^2}f(t,\bx,\bv)\,\frac{\|\bv\|^2}{2b(\bx)}
\,\dd\bx\,\dd\bv\,.
$$
As far as smooth solutions are concerned, the total energy is
preserved by both the original $\eps$-dependent model and by the
asymptotic model \eqref{eq:gc}. One goal of our experimental
observations is to check that despite the fact that our scheme
dissipates some parts of the velocity variable to reach the asymptotic
regime corresponding to \eqref{eq:gc} it does respect this
conservation. In contrast an essentially exact conservation of the
adiabatic invariant is a sign that we have reached the limiting
asymptotic regime since it does not hold for the original model but
does for the asymptotic~\eqref{eq:gc} as $b$ is time-independent and
$\bE$ is curl-free. Observe that, since $b$ is not homogeneous, even in the
asymptotic regime the kinetic and potential parts of the total
energy are not preserved separately, but the total energy
corresponding to the Vlasov-Poisson system is still
preserved. Figure~\ref{fig2:0} shows that all these features are
captured satisfactorily by our scheme even on long time evolutions
with a large time step and despite its dissipative implicit nature
when $\eps=1$ (Figure~\ref{fig2:0} left column) and also in the asymptotic regime when $\eps=0.05$
(Figure~\ref{fig2:0} right column).
  
\begin{figure}[ht!]
\begin{center}
 \begin{tabular}{cc}
\includegraphics[width=7.cm]{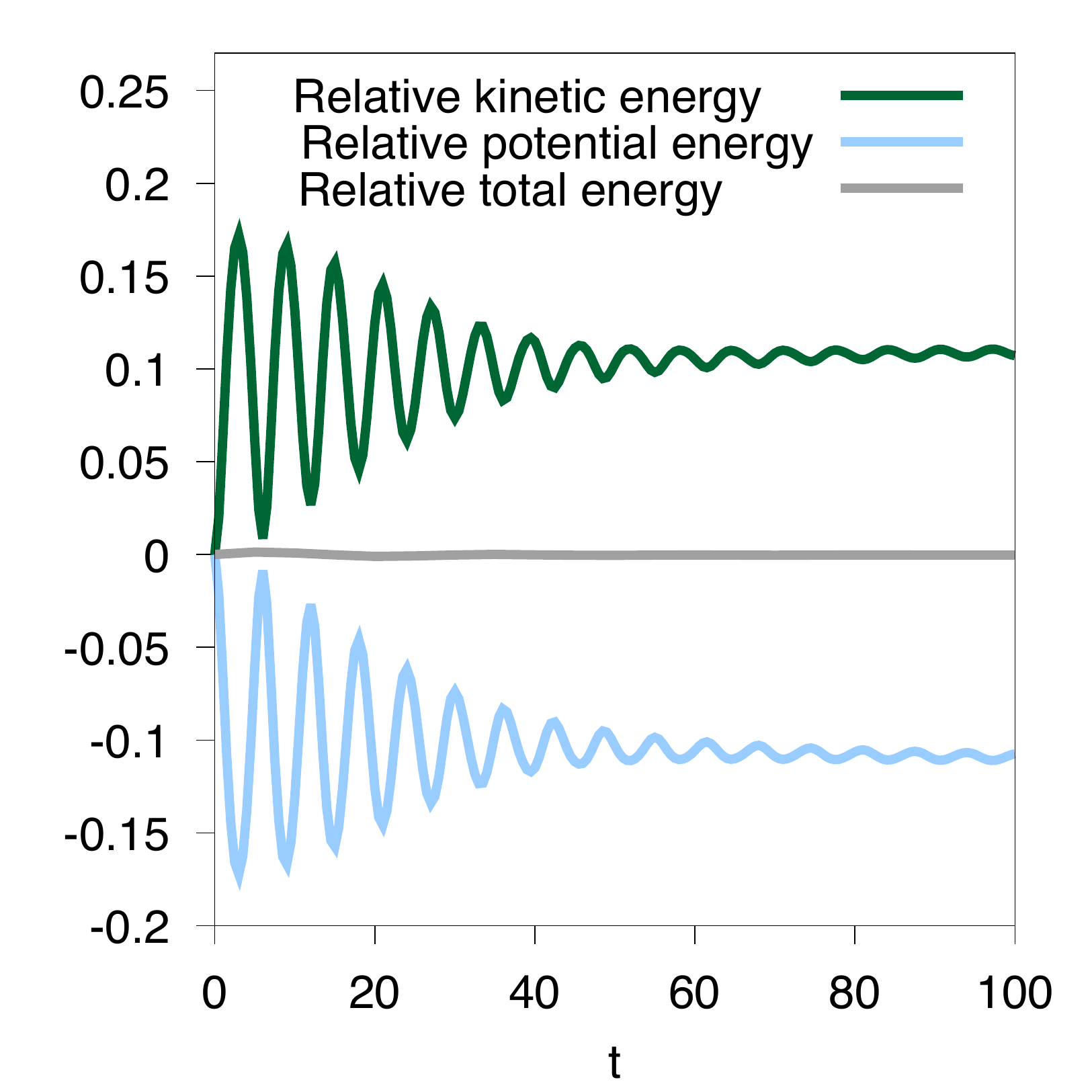} &    
\includegraphics[width=7.cm]{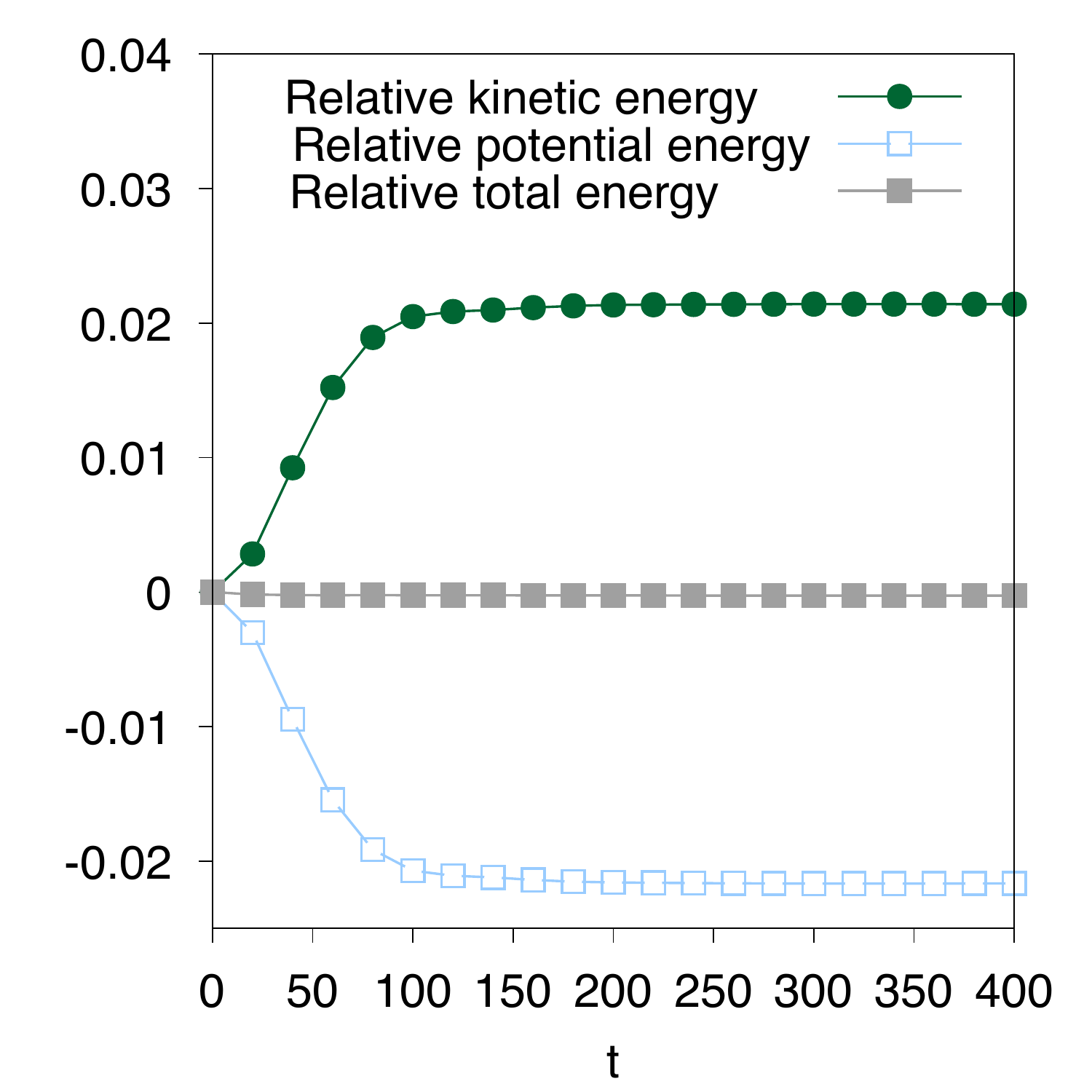} 
\\

\includegraphics[width=7.cm]{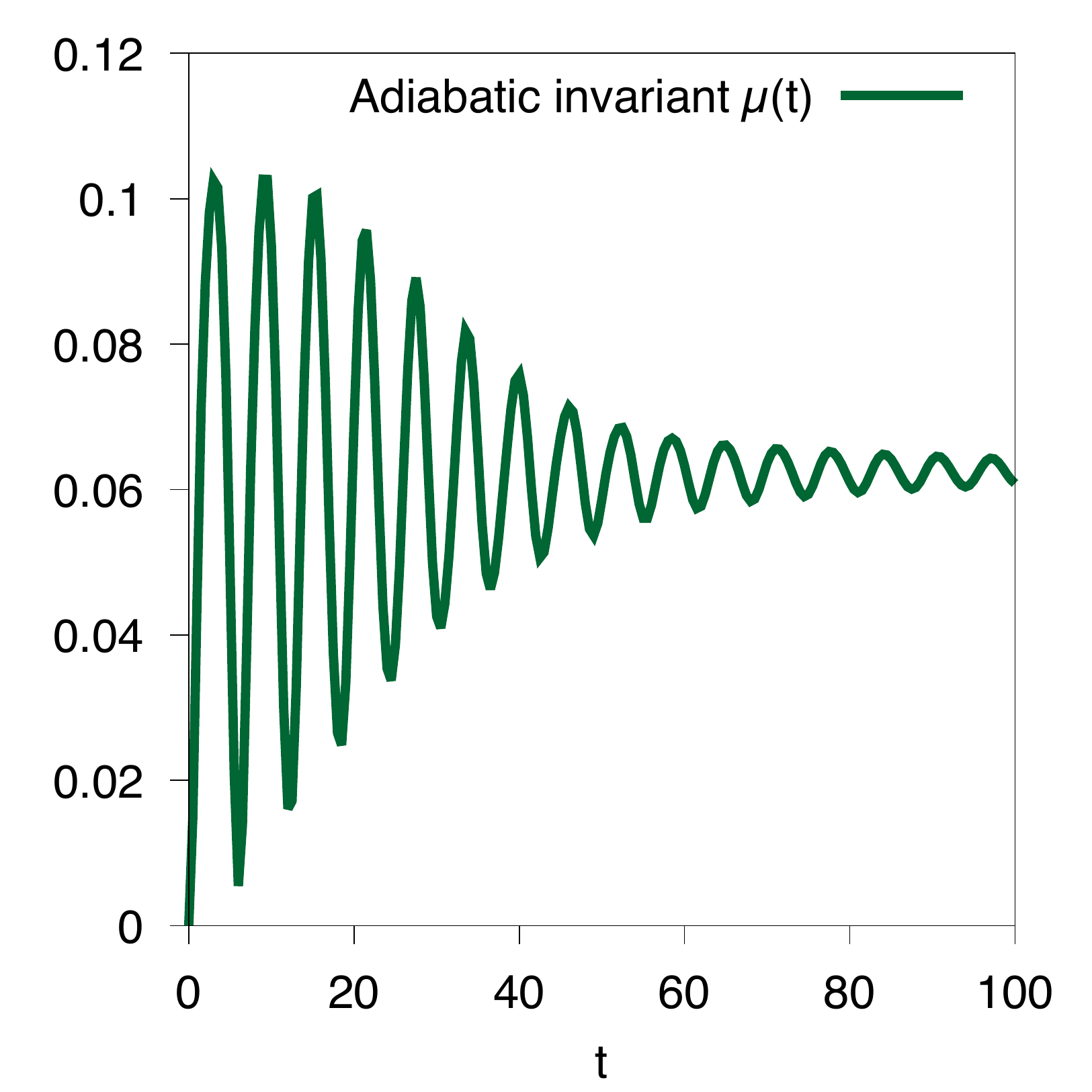} &    
\includegraphics[width=7.cm]{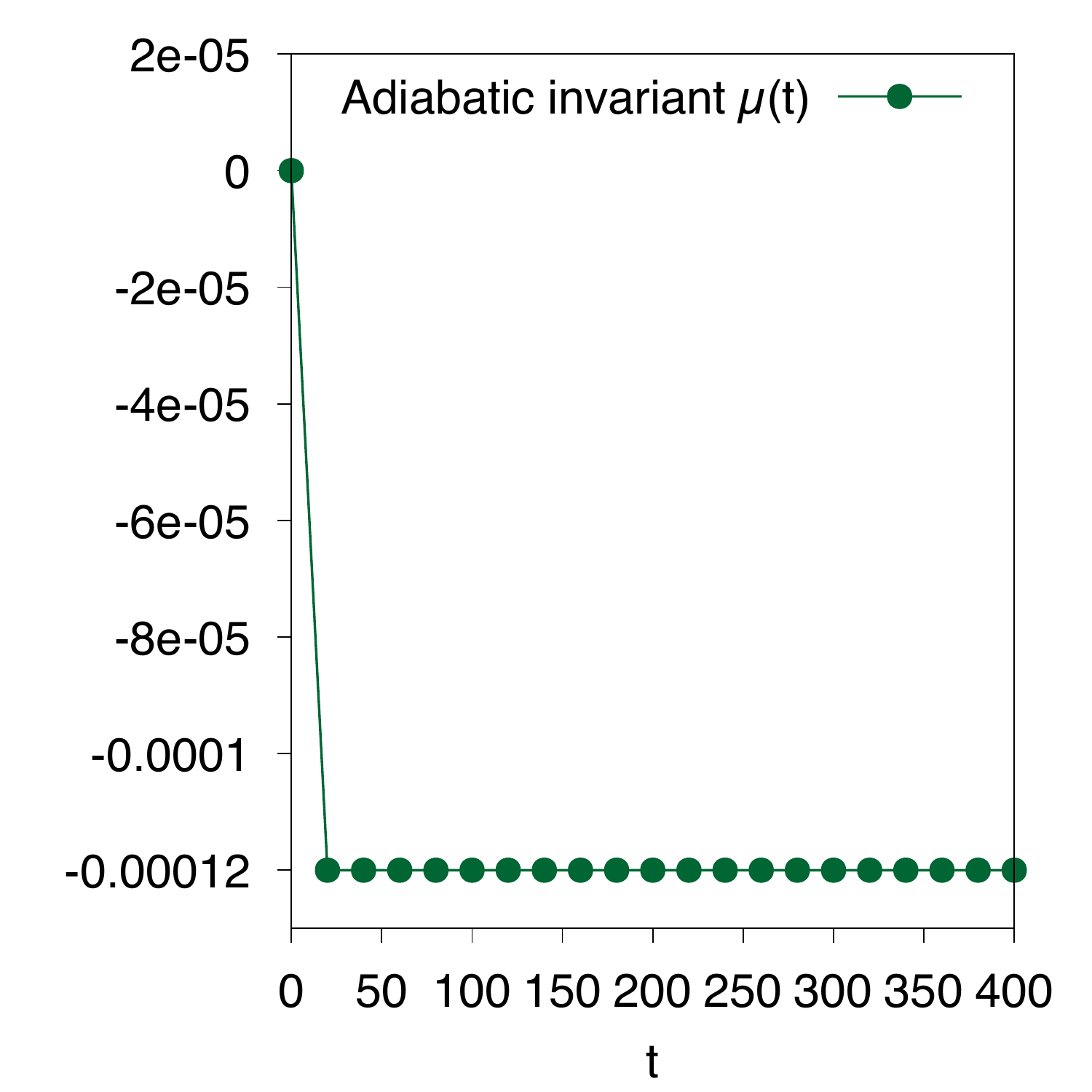} 
\\
(a) $\eps=1$ &(b) $\eps=0.05$
\end{tabular}
\caption{\label{fig2:0}
{\bf The Vlasov-Poisson system.}  Time evolution of total energy and
adiabatic invariant  with (a) $\eps=1$ and (b) $\eps=0.05$ obtained using
\eqref{scheme:4-1}-\eqref{scheme:4-5} with $\Delta t=0.1$.}
 \end{center}
\end{figure}

In Figure~\ref{fig2:1} and \ref{fig2:2} we visualize the corresponding
dynamics by presenting some snapshots of the time evolution of the
macroscopic charge density. In the first case, with $\eps=1$, the magnetic field is not sufficiently large to provide a good confinement of the macroscopic density, and some macroscopic oscillations occur as already observed on the time evolution of macroscopic
quantities. On the other hand, our inhomogeneous case does exhibit some confinement when $\eps\to0$, since the conservation of $e/b(\bx)$ offers (a weak form of) coercivity jointly in
$(\bx,e)$. Such a confinement is indeed observed when $\eps=0.05$, jointly with the expected eventual merging of two initial vortices in a relatively short time. 

\begin{figure}[ht!]
\begin{center}
 \begin{tabular}{cc}
\includegraphics[width=7.cm]{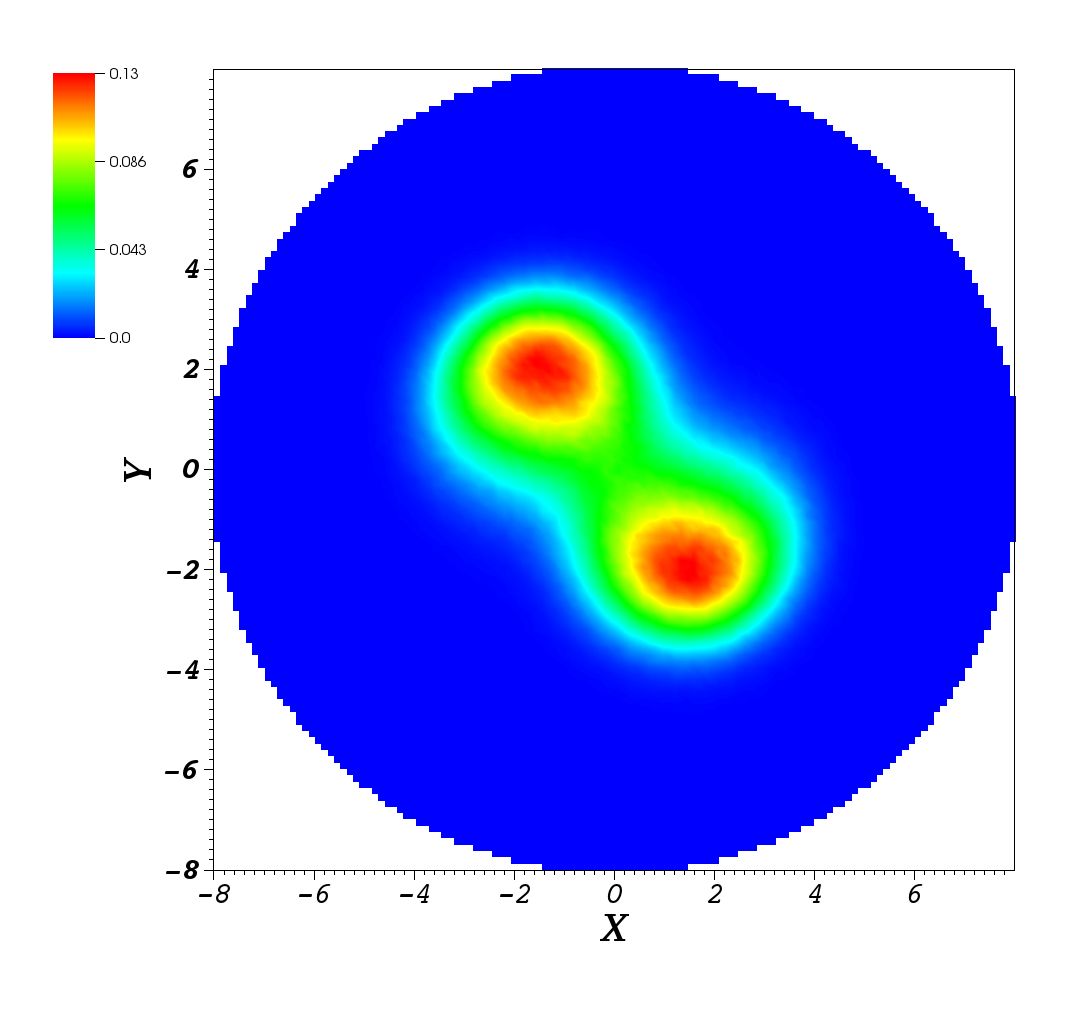} &    
\includegraphics[width=7.cm]{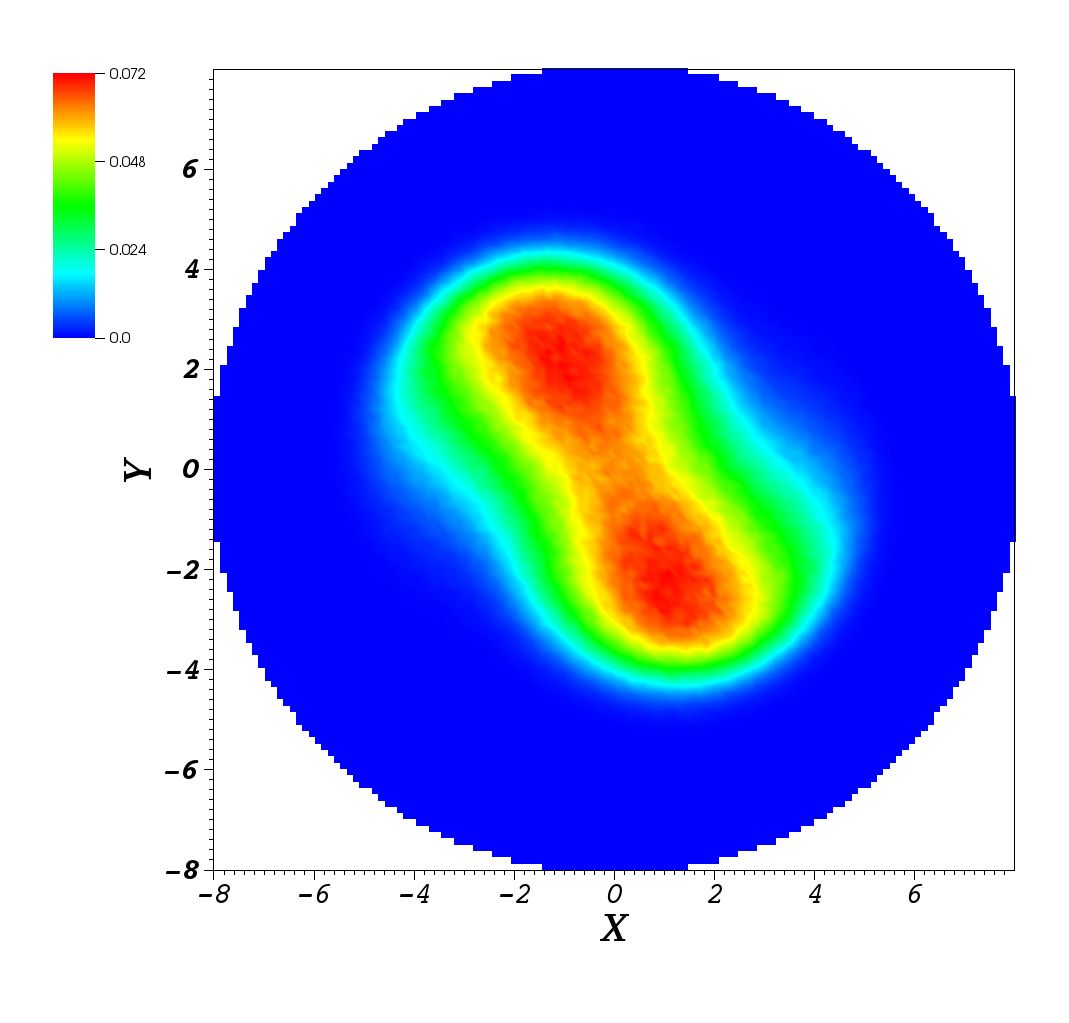} 
\\
$t=10$  & $t=20$ 
\\
\includegraphics[width=7.cm]{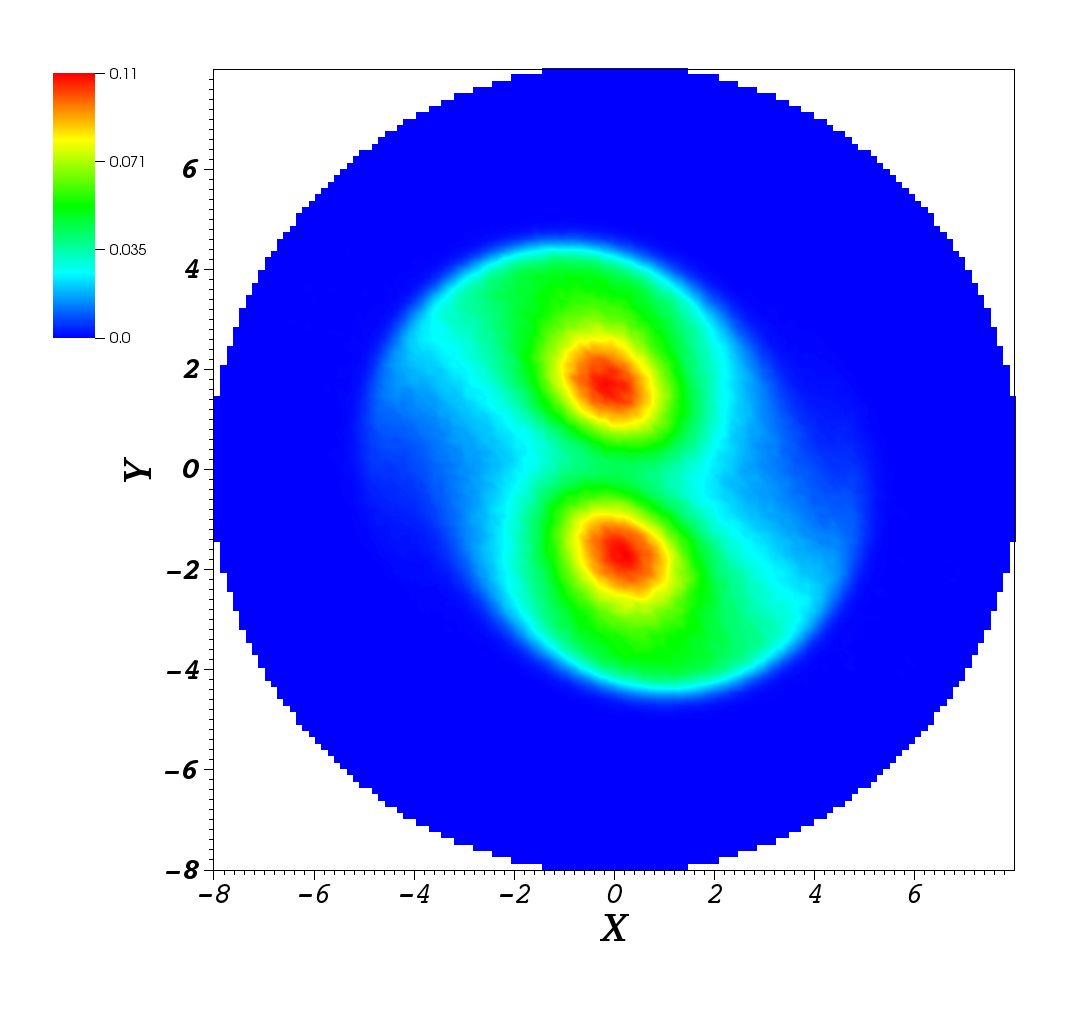} &    
\includegraphics[width=7.cm]{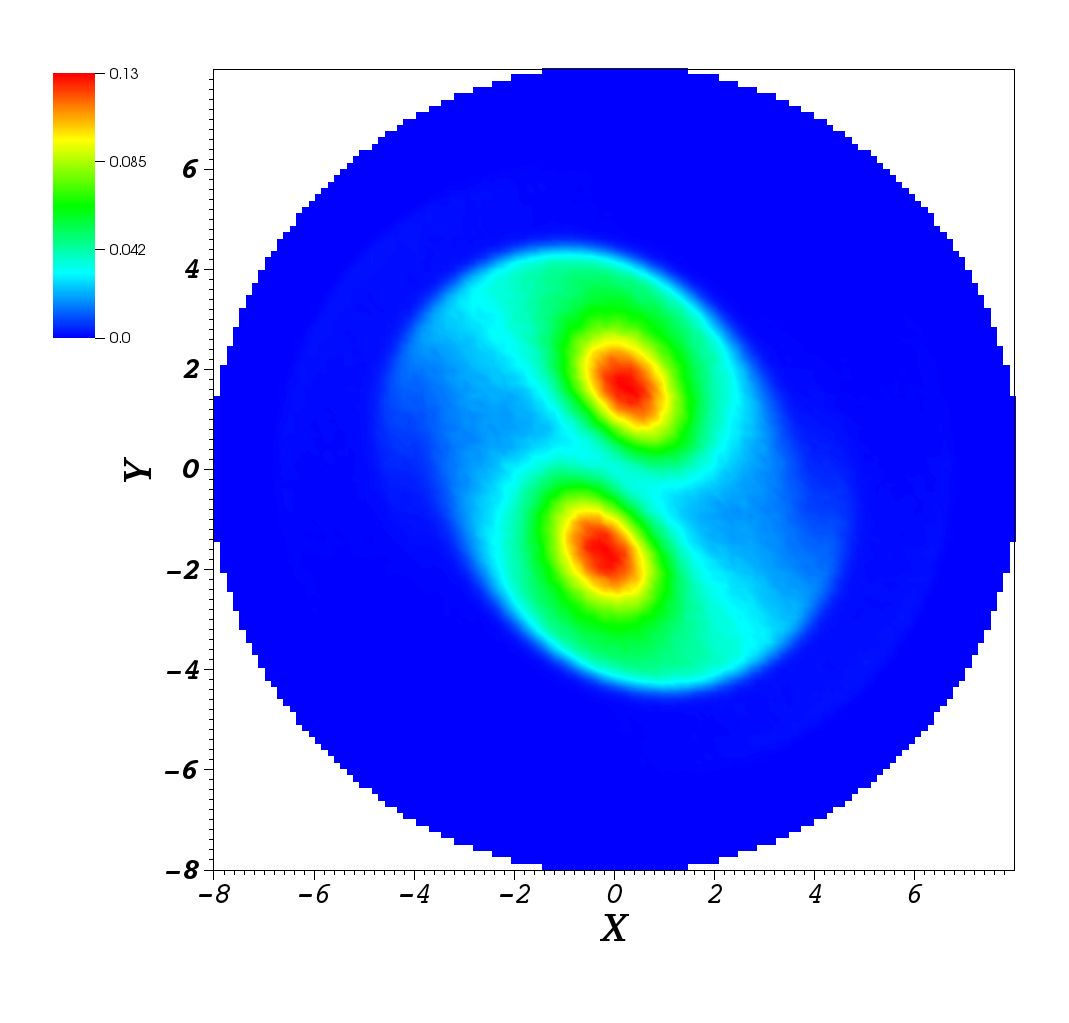} 
\\
$t=55$  & $t=80$ 
\end{tabular}
\caption{\label{fig2:1}
{\bf The Vlasov-Poisson system.}  Snapshots of the time evolution of the macroscopic charge density $\rho$ when $\eps=1$, obtained using
\eqref{scheme:4-1}-\eqref{scheme:4-5} with $\Delta t=0.1$ .}
 \end{center}
\end{figure}

\begin{figure}[ht!]
\begin{center}
 \begin{tabular}{cc}
\includegraphics[width=7.cm]{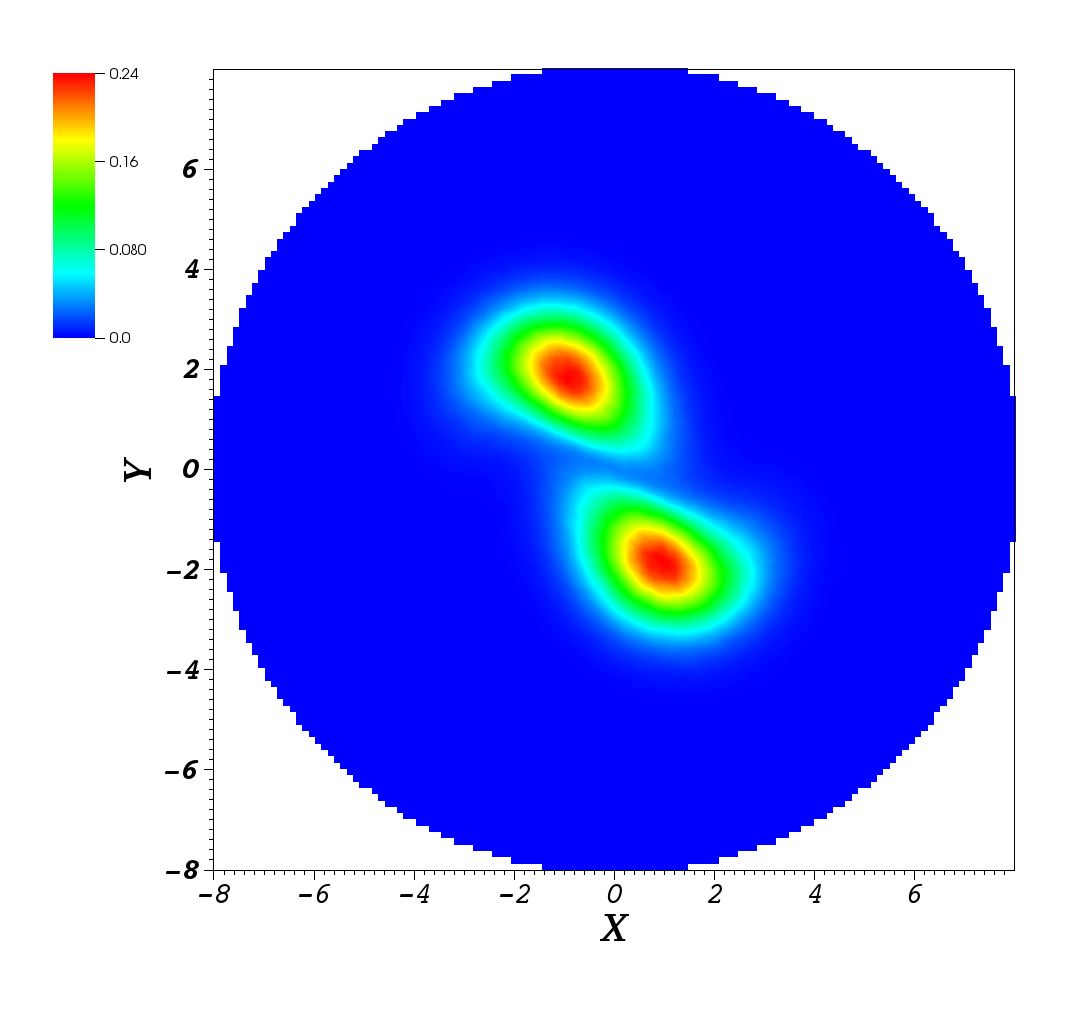} &    
\includegraphics[width=7.cm]{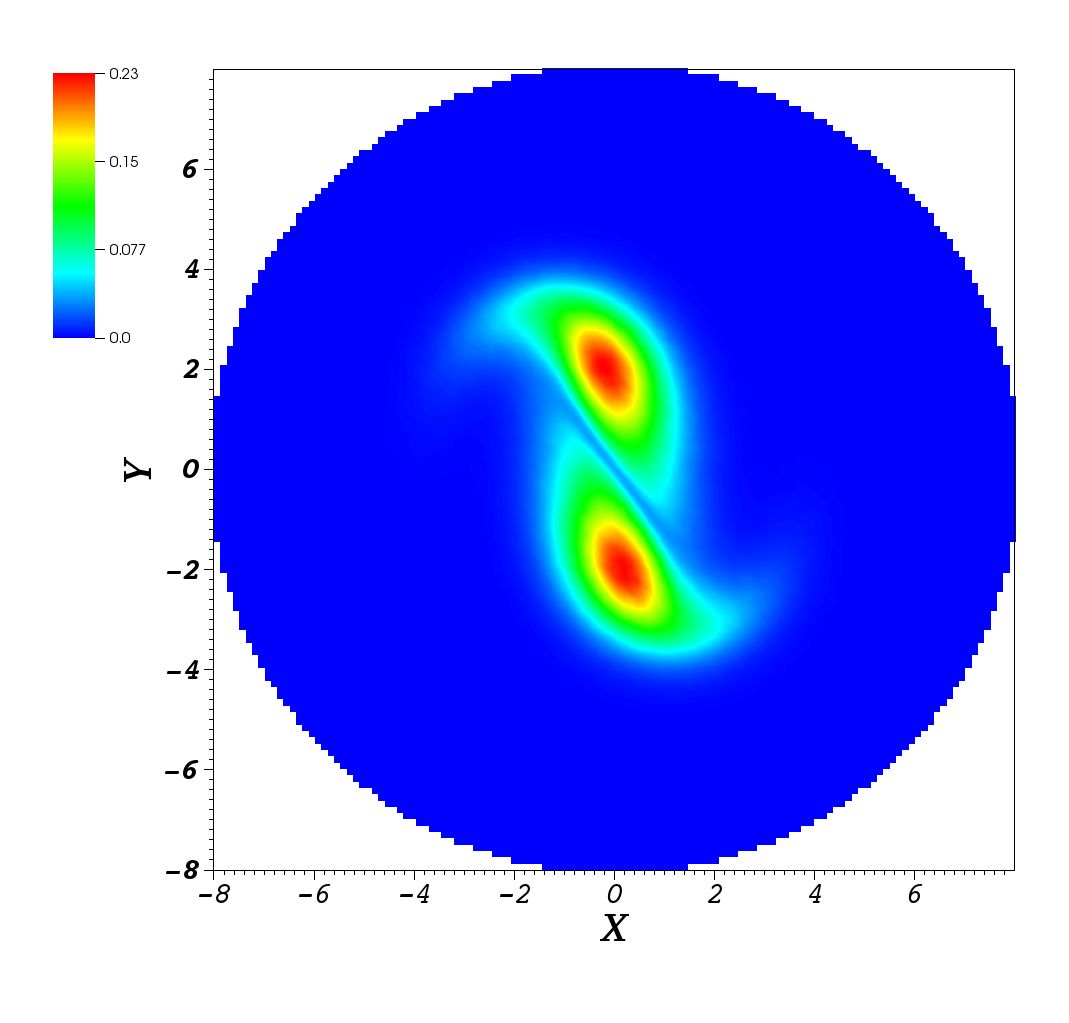} 
\\
$t=10$  & $t=20$ 
\\
\includegraphics[width=7.cm]{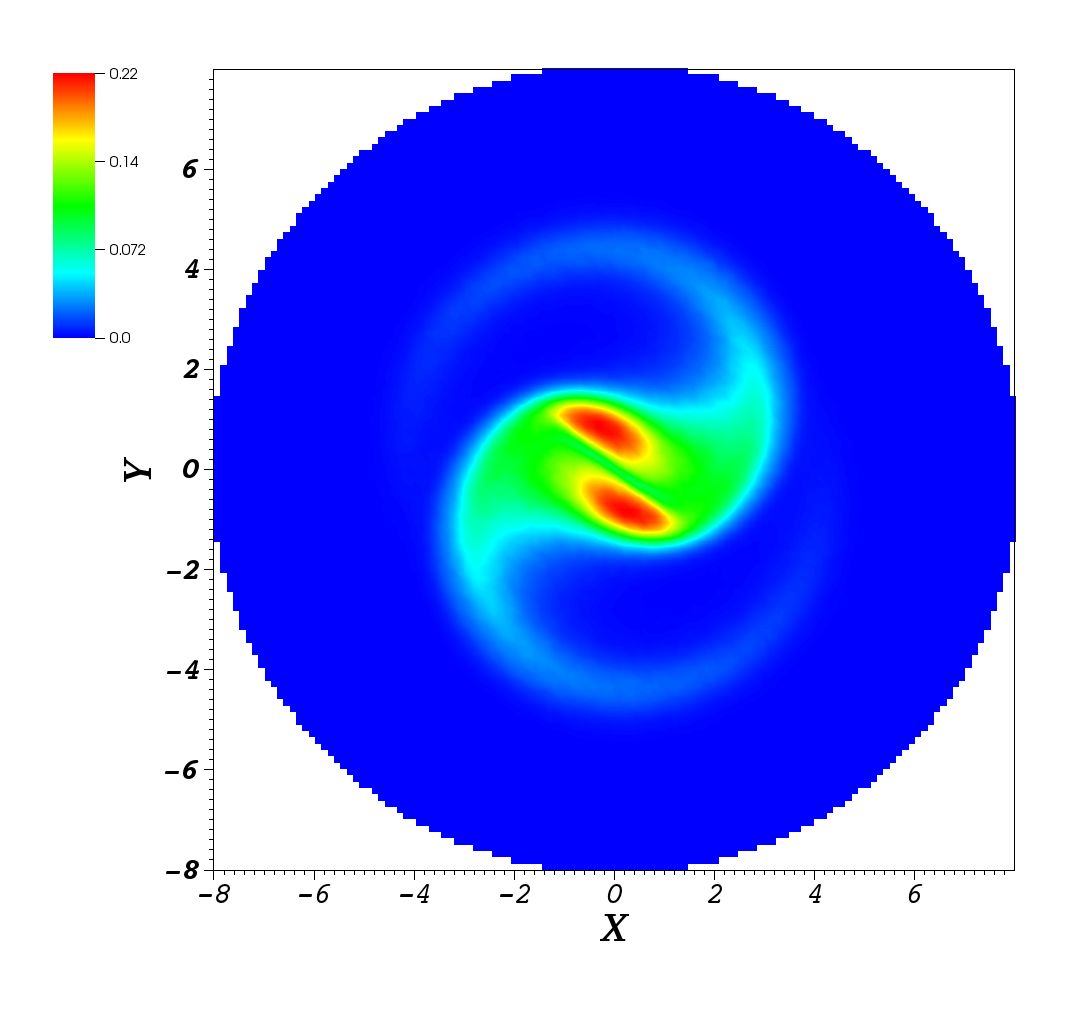} &    
\includegraphics[width=7.cm]{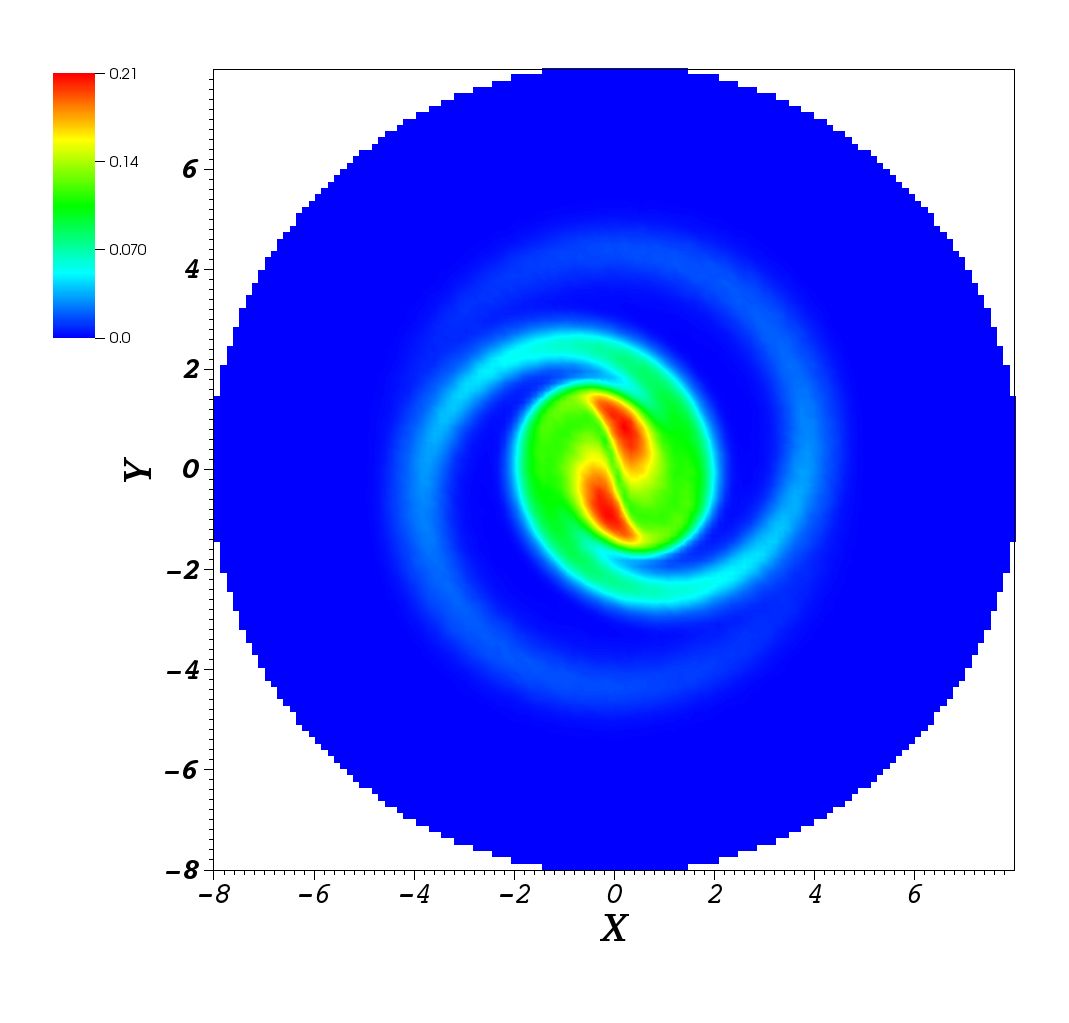} 
\\
$t=55$  & $t=80$ 
\end{tabular}
\caption{\label{fig2:2}
{\bf The Vlasov-Poisson system.}  Snapshots of the time evolution of the macroscopic charge density $\rho$ when $\eps=0.05$, obtained using
\eqref{scheme:4-1}-\eqref{scheme:4-5} with $\Delta t=0.1$ .}
 \end{center}
\end{figure}

\section{Conclusion and perspectives}
\label{sec:6}
\setcounter{equation}{0}
In the present paper we have proposed a class of semi-implicit time discretization
techniques for particle-in cell simulations. The main feature of our approach is to guarantee the accuracy and stability on slow scale variables even when the amplitude of the magnetic field becomes large, thus allowing a capture of their correct long-time behavior including cases with non homogeneous magnetic fields and coarse time grids. Even on large time simulations the obtained numerical schemes also provide an acceptable accuracy on physical invariants (total energy for any $\eps$, adiabatic invariant when $\eps\ll1$) whereas fast scales are automatically filtered when the time step is large compared to $\eps^2$.

As a theoretical validation we have proved that under some stability assumptions on numerical approximations, the slow part of the approximation converges when $\eps\rightarrow 0$ to the solution of a limiting scheme for the asymptotic evolution, that preserves the initial order of accuracy. Yet a full proof of uniform accuracy and a classification of admissible schemes --- going beyond the exposition of a few possible schemes --- remain to be carried out.

From a practical point of view, the next natural step would be to consider the genuinely three-dimensional Vlasov-Poisson system. We point out however that, despite recent progresses \cite{Cheverry,PDFF}, the asymptotic limit $\eps\to0$ is much less understood in this context, even at the continuous level with external electro-magnetic fields.

\section*{Acknowledgements}
FF was supported by the EUROfusion Consortium and has received funding
from the Euratom research and training programme 2014-2018 under grant
agreement No 633053. The views and opinions expressed herein do not
necessarily reflect those of the European Commission.

LMR was supported in part by the ANR project BoND (ANR-13-BS01-0009-01). LMR also expresses his appreciation of the hospitality of IMT, Universit\'e Toulouse III, during part of the preparation of the present contribution.

\bibliography{paper.bib}
\bibliographystyle{abbrv}

\begin{flushleft} \signFF \end{flushleft}
\begin{flushright} \signLMR \end{flushright}

\end{document}